\documentclass[11pt]{amsart}

\newcommand{\private}[1]{}
\usepackage{amssymb, amsmath, amscd, amsthm, color, epsfig,url}

\usepackage[all]{xy}          
\xyoption{dvips}              

\makeatletter
\renewcommand\l@subsection{\@tocline{2}{0pt}{2pc}{5pc}{}}
\makeatother

\newcommand{\R}{{\mathbb R}}

\newcommand{\abs}[1]{{\left\vert #1 \right\vert}}

\newcommand{\hofiber}{\operatorname{hofiber}}

\newcommand{\holim}{\operatorname{holim}}

\newcommand{\tfiber}{\operatorname{tfiber}}

\newcommand{\Map}{\operatorname{Map}}

\newcommand{\Emb}{\operatorname{Emb}}

\newcommand{\Imm}{\operatorname{Imm}}
\newcommand{\Link}{\operatorname{Link}}

\newcommand{\Spaces}{\operatorname{Spaces}}

\newcommand{\K}{{\mathcal{K}}}
\newcommand{\LK}{{\mathcal{LK}}}
\newcommand{\HLK}{{\mathcal{HLK}}}
\newcommand{\BR}{{\mathcal{BR}}}

\newcommand{\Xdot}{{X^{\bullet}}}
\newcommand{\Kdot}{{K^{\bullet}}}

\newcommand{\del}{{\partial}}

\newcommand{\Top}{\operatorname{Top}}

\newcommand{\Tot}{\operatorname{Tot}}

\theoremstyle{plain}
\newtheorem{thm}{Theorem}[section]
\newtheorem{prop}[thm]{Proposition}
\newtheorem{lemma}[thm]{Lemma}
\newtheorem{cor}[thm]{Corollary}

\theoremstyle{definition}
\newtheorem{defin}[thm]{Definition}
\newtheorem{example}[thm]{Example}
\newtheorem{def/ex}[thm]{Definition/Example}

\theoremstyle{remark}
\newtheorem{rem}[thm]{Remark}
\newtheorem{rems}[thm]{Remarks}

\newcommand{\refS}[1]{Section~\ref{S:#1}}
\newcommand{\refT}[1]{Theorem~\ref{T:#1}}
\newcommand{\refC}[1]{Corollary~\ref{C:#1}}
\newcommand{\refP}[1]{Proposition~\ref{P:#1}}
\newcommand{\refD}[1]{Definition~\ref{D:#1}}
\newcommand{\refL}[1]{Lemma~\ref{L:#1}}
\newcommand{\refE}[1]{equation~$(\ref{E:#1})$}

\begin{document}


\title[Cosimplicial models for spaces of links]{Cosimplicial models for spaces of links}


\author{Brian A. Munson}
\address{Department of Mathematics, Wellesley College, Wellesley, MA}
\email{bmunson@wellesley.edu}
\urladdr{http://palmer.wellesley.edu/\~{}munson}

\author{Ismar Voli\'c}
\address{Department of Mathematics, Wellesley College, Wellesley, MA}
\email{ivolic@wellesley.edu}
\urladdr{http://palmer.wellesley.edu/\~{}ivolic}


\thanks{The second author was supported in part by the National Science Foundation grant DMS 0805406.}


\begin{abstract}
We study the spaces of string links and homotopy string links in an arbitrary manifold using multivariable manifold calculus of functors.   We construct multi-cosimplicial models for both spaces and  deduce certain convergence properties of the associated Bousfield-Kan homotopy and cohomology spectral sequences when the ambient manifold is a  Euclidean space of dimension four or more.  
\end{abstract}

\maketitle

\tableofcontents

\parskip=4pt
\parindent=0cm


\section{Introduction}\label{S:Intro}


Let $\LK_m(\vec{I},I^n)$ and $\HLK_m(\vec{I},I^n)$ be spaces of embeddings and link maps (links and homotopy links) of $m$ copies of the interval $I$ in $I^n$ which are fixed at the boundary (see \refD{linkspaces}).  The goal of this paper is to define cosimplicial models for these spaces as well as study the Bousfield-Kan homotopy and cohomology spectral sequences associated to them.

We begin with the application of multivariable manifold calculus of functors \cite{MV:Multi} to $\LK_m(\vec{I},I^n)$ and $\HLK_m(\vec{I},I^n)$.  This theory assigns a certain multi-tower of approximations $T_{\vec{\jmath}}\LK_m(\vec{I},I^n)$ and $T_{\vec{\jmath}}\HLK_m(\vec{I},I^n)$ to these spaces.  In the case of 
 $\LK_m(\vec{I},I^n)$, this multi-tower converges to the link space if $n$ is four or more (\refT{MultiConvergence}).  Nothing is known about the convergence in the case of $\HLK_m(\vec{I},I^n)$.  The stages $T_{\vec{\jmath}}$ of the multi-tower can be realized as homotopy limits of cubical diagrams of punctured links (\refP{FiniteTower}) which are homotopy equvalent to (partial) configuration spaces (Propositions \ref{P:PuncturedLinksConfs} and \ref{P:PuncturedHomotopyLinksConfs}).  The maps in the cubical diagrams in effect double configuration points, but the best way to make this precise is to pass to the cosimplicial model for the multi-tower.
 
The passage from cubical diagrams to cosimplicial spaces, explained in \refS{CosimplicialCubical}, allows us to define multi-cosimplicial spaces $LK^{\vec{\bullet}}_m(\vec{I},I^n)$ and $HLK^{\vec{\bullet}}_m(\vec{I},I^n)$ whose totalizations (or homotopy limits) are precisely the homotopy limits of the multi-towers for $\LK_m(\vec{I},I^n)$ and $\HLK_m(\vec{I},I^n)$, respectively (Propositions \ref{P:MultiStagesCosimplicialModel} and \ref{P:HMultiStagesCosimplicialModel}).  

We next consider the diagonal cosimplicial spaces $(LK^{\vec{\bullet}}_m)_{diag}(\vec{I},I^n)$ and $(HLK^{\vec{\bullet}}_m)_{diag}(\vec{I},I^n)$ which have the same totalization as the original multi-cosimplicial spaces and to which we can readily associate homotopy and cohomology Bousfield-Kan spectral sequences.  The bulk of this paper has to do with the convergence properties of these spectral sequences for $n>3$.
Here is a summary of the main results:
\begin{itemize}
\item  There exist homotopy and (co)homology spectral sequences which start with the homotopy and (co)homology groups of configuration spaces and which converge to $\LK_m(\vec{I},I^n)$ for $n>3$ (\refT{MainThmLinks} and \refC{ConvergenceToLinks}). 
\item  There exist homotopy and (co)homology spectral sequences which start with the homotopy and (co)homology groups of certain partial configuration spaces and which converge to $\Tot(HLK^{\bullet}_m)_{diag}(\vec{I},I^n)\simeq \Tot HLK^{\bullet}_m(\vec{I},I^n)$ for $n>3$ (\refT{MainThmHomotopyLinks}). 
\end{itemize}

Much of these results are generalizations of the existing results for the space of long knots $\K(I, I^n)$ from \cite{S:TSK}.
Whenever possible, we will throughout this paper direct the reader to the appropriate places in Sinha's work where corresponding statements are to be found.  The novelty in this paper is two-fold:  (1) Multivariable manifold calculus and multicosimplicial spaces have been brought into the picture;  (2) This is the first time that homotopy string links (and braids) have been studied by these methods.  In addition, in Section \ref{S:CosimplicialReview} we give a self-contained, detailed exposition of the relationship between cosimplicial spaces and cubical diagrams which is of independent interest.  This part elaborates upon and provides details of some of the results and techniques used in \cite{S:TSK}. 

One immediate issue, and avenue of future investigation, is that much less is known about homotopy links from the point of view of calculus of functors than it is about ordinary links.  While it is not surprising that links, which generalize knots, lend themselves to an examination via calculus of functors, that the same theory applies to homotopy links turns out to be potentially very fruitful, although there is still much to do in understanding their multivariable approximations $T_{\vec{\jmath}}$.  Some of the authors' original motivation for writing this paper has to do with the more general investigation of spaces  of link maps, of which homotopy links $\HLK_m(\vec{I}, I^n)$ are examples, via calculus of functors.  More details are given below in the outline of future work.

The paper is organized as follows:  In Sections \ref{S:MultiCalc} and \ref{S:CosimplicialReview} we give the necessary background on multivariable calculus and cosimplicial spaces, including Bousfield-Kan spectral sequences.  In \refS{CosimplicialModel}, we first give a review of Sinha's cosimplicial model for long knots (\refS{KnotsModel}) and then generalize it to links (\refS{LinksModel}) and homotopy links (\refS{HomotopyLinksModel}).  \refS{LinksSS's} contains the main results about the convergence of spectral sequences associated to our cosimplicial models.  Lastly, in \refS{Braids}, we make an observation that our constructions apply equally well to the space of braids and that this gives the familiar cobar construction on the loops of configuration spaces.  

Here are some of the possible directions of future work following the results established in this paper.

\begin{itemize}

\item \refT{MainThmHomotopyLinks} says that the cohomology spectral sequence for  $(HLK^{\bullet}_m)_{diag}(\vec{I},I^n)$ converges to its totalization for $n>3$.  For $(LK^{\bullet}_m)_{diag}(\vec{I},I^n)$, we also have that the homotopy spectral sequence converges to its totalization for $n> 3$ and, moreover, this totalization is precisely the space of links $\LK_m(\vec{I},I^n)$ in that range.  A brief summary of where this difference between links and homotopy links comes from can be found in Remark \ref{R:RemHLinksConvergence}; in particular, we do not have analogs of \refT{MultiConvergence} or \refP{HomotopyConvergence} for homotopy string links.  Establishing these results for homotopy links would be very beneficial, although this appears to be a difficult problem.

\item Along the same lines, when $n=3$, nothing is known about the convergence of the multi-tower or the spectral sequences for either $\LK_m(\vec{I},I^3)$ or $\HLK_m(\vec{I},I^3)$.  This is also the case for long knots $\K(I, I^3)$.  Nevertheless, in the case of long knots, the manifold calculus tower classifies finite type knot invariants \cite{V:FTK} and this result should extend easily to spaces of links and braids using cosimplicial models defined in this paper.  The goal then is to reprove that finite type invariants separate braids and homotopy links and perhaps gain insight into the problem of separation of ordinary knots and links by finite type invariants.  This should further lead into the discovery of Milnor invariants in the manifold calculus multi-tower and a potential alternative proof of the Habegger-Lin classification of homotopy string links in $I^3$ \cite{HabLin-Classif}.

\item It was shown in \cite{LTV:Vass} that the cohomology Bousfield-Kan spectral sequence associated to $\K(I, I^n)$ collapses rationally for $n>3$.  It seems likely that this result generalizes to the spectral sequences for $\LK_m(\vec{I},I^n)$ and $\HLK_m(\vec{I},I^n)$ defined in this paper. 



\end{itemize}


\subsection{Acknowledgements}


The authors would like to thank Phil Hirschhorn for help with the reference for \refP{HomotopyConvergenceCondition} and to Ryan Budney and Paolo Salvatore for pointing out an error in an earlier version of this paper.


\section{Conventions and notation}\label{S:Conventions}


Throughout the paper, we will assume the reader is familiar with the basic language of homotopy (co)limits.
We work in the category of topological spaces and denote it by $\Spaces$.  We also write  $\Spaces_\ast$ for the category of based spaces. Other conventions are as follows.

\vskip 6pt

\textbullet\ \ 
For a nonnegative integer $p$, let $[p]$ denote the set $\{0,1,\ldots, p\}$, and let $\underline{p}$ denote the set $\{1,2,\ldots, p\}$. 

\textbullet\ \ 
For a finite set $S$, we let $\abs{S}$ stand for its cardinality. For a tuple $\vec{S}=(S_1,\ldots, S_m)$ of sets, we let $|\vec{S}|=\sum_i |S_i|$.

\textbullet\ \ 
Let $\mathcal{P}(S)$ stand for the poset of all subsets of $S$, and $\mathcal{P}_0(S)$ the subposet of all non-empty subsets of $S$. For a tuple $\vec{S}=(S_1,\ldots, S_m)$ of finite sets, let $\mathcal{P}(\vec{S})=\prod_{i=1}^m\mathcal{P}(S_i)$.

\textbullet\ \ 
For tuples $\vec{\jmath}=(j_1,j_2,\ldots, j_m)$ and $\vec{k}=(k_1,k_2,\ldots, k_m)$ of integers, we say $\vec{\jmath}\leq\vec{k}$ if $j_i\leq k_i$ for all $1\leq i\leq m$, and $\vec{\jmath}<\vec{k}$ if $\vec{\jmath}\leq\vec{k}$ and there exists $i$ such that $j_i<k_i$. 

\textbullet\ \ 
Let $|\vec{\jmath}|=\sum_i\abs{j_i}$, and $|\vec{\jmath}-\vec{k}|=\sum_i\abs{j_i-k_i}$. 



\textbullet\ \ 
For a space $X$ and a nonnegative integer $j$, suppose $S\subset\underline{j}$ is nonempty. Let $\Delta_S(X^j)\subset X^j$ be defined by
$$\Delta_S(X^j)=\{(x_1,\ldots, x_j)\in X^j\, |\, x_i=x_k\mbox{ for all } i,k\in S\}.$$ 

\textbullet\ \ 
For a space $X$ and a nonnegative integer $j$, let $C(j,X)$ be the \emph{configuration space} of $j$ points in $X$, in other words 
$$
C(j,X)=X^j-\cup_{|S|=2} \Delta_S(X^j).
$$ 
Given a partition of $j$, that is, $j=p_1+\cdots +p_m$ for nonnegative integers $p_i$, it will be convenient at times for us to write $C(p_1+\cdots +p_m,X)$ or $C(\sum p_i,X)$ in place of $C(j,X)$. We need a way to view these configuration spaces as being part of an $m$-cosimplicial space (see \refS{m-Cosimplicial}), and each of our configuration points will be associated with one of $m$ distinct intervals.

\textbullet\ \ Let $X=N$ and choose points $x_0, x_{j+1}\in \partial N$.  We will denote by $C_{\partial}(j.N)$ the configuration space of $j+2$ points in $N$ such that first and last are $x_0$ and $x_{j+1}$.  Thus 
$$C(j,N)\simeq C_{\partial}(j,N)\hookrightarrow C(j+2, N).
$$

\textbullet\ \ 
For a space $X$ and a tuple $\vec{\jmath}=(j_1,\ldots, j_m)$ of nonnegative integers, let $X^{\vec{\jmath}}=X^{j_1}\times\cdots\times X^{j_m}$. Write elements of $X^{\vec{\jmath}}$ as tuples $(\vec{x}_1,\ldots, \vec{x}_m)$, where $\vec{x}_i=(x^i_1,\ldots, x^i_{j_i})\in X^{j_i}$.

\textbullet\ \ 
Let $\vec{S}=(S_1,\ldots, S_m)$ be a tuple of subsets such that $|S_i|\leq 1$ and $S_i\subset \underline{j_i}$ for all $i$. Thus $S_i=\{k_i\}$ for some $k_i\in\underline{j_i}$. Define 
$$
\Delta_{\vec{S}}(X^{\vec{\jmath}})=\{(\vec{x}_1,\ldots, \vec{x}_m)\ |\  x^i_{k_i}=x^l_{k_l}\mbox{ for all } i,l\mbox{ such that } S_i,S_l\neq\emptyset\}.
$$ Define the \emph{partial configuration space}
$$
C(\vec{\jmath},X)=X^{\vec{\jmath}}-\underset{|\vec{S}|=2}{\cup}\Delta_{\vec{S}}(X^{\vec{\jmath}}).
$$
It will be useful at times to write $C(j_1,\ldots, j_m;X)$ in place of $C(\vec{\jmath},X)$. We call these \emph{partial} configuration spaces since not all possible diagonals in $X^{\vec{\jmath}}$ have been removed. This space can be thought of as obtained from $X^{\sum_{i=1}^{m} j_i}$ by removing some of the diagonals and is in the complement of a subspace arrangement.  Such spaces have been studied extensively.  We will in particular be interested in its cohomology generators in Section \ref{S:LinksSS's}.

\textbullet\ \ When $X=N$ is a manifold with boundary, fix pairs $x_0^i$ and $x_{j_i+1}^i$ for all $i=1$ to $m$, all distinct in $\del N$. Write $C_\del({\vec{\jmath}},N)$ for the space of partial configurations with the additional data that, for all $i$, the first and last point in the $i^{th}$ variable are $x_0^i$ and $x_{j_i+1}^i$ respectively. Thus, 
$$
C({\vec{\jmath}},N)\simeq C_\del({\vec{\jmath}},N)\hookrightarrow C({\vec{\jmath}+\vec{2}},N).
$$

\section{Multivariable manifold calculus and spaces of links}\label{S:MultiCalc}


Let $P=\coprod_i P_i$ be a finite disjoint union of $m$ smooth closed manifolds of possibly different dimensions. Let $\mathcal{O}(P)$ denote the poset of open subsets of $P$. Multivariable manifold calculus \cite{MV:Multi}, generalizing single variable manifold calculus \cite{W:EI1}, studies contravariant functors $F:\mathcal{O}(P)\rightarrow \Spaces$ (satisfying certain technical conditions; see \cite[Definition 4.2]{MV:Multi}).  The goal of both theories is to approximate $F$ by certain polynomial functors defined below.  We have an equivalence of categories 
$$
\mathcal{O}\left(\coprod_iP_i\right)\cong\prod_i\mathcal{O}(P_i)
$$
given by sending $U=\coprod U_i\in\mathcal{O}(P)$, where $U_i=U\cap P_i$, to $\vec{U}=(U_1,\ldots, U_m)$. This gives two ways of viewing a functor $F$ defined on this category. On the one hand, when we view the domain as the category $\mathcal{O}(\coprod_iP_i)$, we think of $F$ as a function of a single variable $U\in\mathcal{O}(\coprod_iP_i)$. On the other hand, the category $\prod_i\mathcal{O}(P_i)=\mathcal{O}(\vec{P})$ has objects $\vec{U}=(U_1,\ldots, U_m)$ and $F$ can be regarded as a functor in each variable separately. Thus we will write either $F(\coprod_i U_i)$ or $F(\vec{U})$ for the values of our functor on an open set $U$. The main examples are all in one way or another related to the following three.

\begin{defin}\label{D:somefunctors}
Let $P$ be as in the previous paragraph, let $N$ be a smooth manifold, and let $U=\coprod_i U_i\in\mathcal{O}(P)$ (also written $\vec{U}=(U_1,\ldots, U_m)$). Define
\begin{enumerate}
\item $\vec{U}\mapsto\Imm(\vec{U},N)$ to be the space of immersions of $U=\coprod_i U_i$ in $N$;
\item $\vec{U}\mapsto\Emb(\vec{U},N)$ to be the space of embeddings of $U=\coprod_i U_i$ in $N$; and
\item $\vec{U}\mapsto\Link(\vec{U},N)$ to be the space of link maps of $U=\coprod_i U_i$ in $N$. This is the space of smooth maps $U\rightarrow N$ such that the images of $U_i$ and $U_j$ are disjoint for all $i\neq j$.
\end{enumerate}
\end{defin}

We can extend to the case where the $P_i$ have boundary $\del P_i$ as well. Here $\mathcal{O}(P)$ is the poset of open subsets of $P$ which contain the boundary. We focus on the examples above in extending to this case. Suppose $N$ is a manifold with boundary $\del N$, and that we are given a neat embedding of $P=\coprod_iP_i$ in $N$. That is, we are given an embedding $e:P\rightarrow N$ such that $e^{-1}(\del N)=\del P$, and $e$ meets $\del N$ transversely.

\begin{defin}\label{D:someexampleswithboundary}
Define $\vec{U}\mapsto\Imm_\del(\vec{U},N)$, $\vec{U}\mapsto\Emb_\del(\vec{U},N)$, and $\vec{U}\mapsto\Link_\del(\vec{U},N)$ to be the same as above, but each element of these spaces is required to agree with the embedding $e$ near $\del P$.
\end{defin}

We now recall the definition of the Taylor multi-tower for a functor $F:\mathcal{O}(\vec{P})\rightarrow\Spaces$.  As above, let $\vec{U}=(U_1\ldots, U_m)\in\mathcal{O}(\vec{P})$ and let $\vec{\jmath}=(j_1,\ldots, j_m)$ be an $m$-tuple of integers with $j_i\geq -1$. For all $1\leq i\leq m$, let $A_0^i,\ldots, A_{j_i}^i$ be pairwise disjoint closed subsets of $U_i$ (if $j_i=-1$, the collection is empty for that $i$). For each $i$, let $A^i_{S_i}=\cup_{j\in S_i}A_{j}^i$, and $\vec{A}_{\vec{S}}=(A^1_{S_1},\ldots, A^m_{S_m})$, where $\vec{S}=(S_1,S_2,\ldots S_m)$ denotes an element of $\mathcal{P}([j_1])\times\cdots\times\mathcal{P}([j_m])$. Put $\vec{U}-\vec{A}_{\vec{S}}=(U_1-A^1_{S_1},\ldots, U_m-A^m_{S_m})$.

\begin{defin}[\cite{MV:Multi}, Definition 4.3]\label{D:Multipoly}
A functor $F:\mathcal{O}(\vec{P})\rightarrow\Spaces$ is said to be \textsl{polynomial of degree $\leq\vec{\jmath}=(j_1,\ldots, j_m)$} if, for all $\vec{U}=(U_1,\ldots, U_m)\in\mathcal{O}(\vec{P})$ and for all pairwise disjoint closed subsets $A_0^i,\ldots, A_{j_i}^i$ in $U_i$, $1\leq i\leq m$, the map $$
F(\vec{U})\longrightarrow\underset{\vec{S}\neq\vec{\emptyset}}{\holim}\, \, F(\vec{U}-\vec{A}_{\vec{S}})
$$
is an equivalence. If $\vec{\jmath}$ has the property that $j_l=-1$ for all $l\neq i$, then we say that F is \emph{polynomial of degree $\leq j_i$ in the $i^{th}$ variable}.
\end{defin}

\begin{example}
The functor $\vec{U}\mapsto\Imm(\vec{U},N)$ is a polynomial of degree $\leq 1$ in each variable. Neither $\vec{U}\mapsto\Emb(\vec{U},N)$ nor $\vec{U}\mapsto\Link(\vec{U},N)$ is polynomial of degree $\leq\vec{\jmath}$ for any $\vec{\jmath}$. See \cite[Section 4.2]{MV:Multi} for some other examples.
\end{example}

\begin{defin}[\cite{MV:Multi}, Definition 4.10]\label{D:jMultiStage}  Define the \emph{$\vec{\jmath}\;^{th}$ degree Taylor approximation of $F$} or \emph{$\vec{\jmath}\;^{th}$ stage of the Taylor multi-tower of $F$} to be
$$T_{\vec{\jmath}}F(\vec{U})=\underset{\mathcal{O}_{\vec{\jmath}}(\vec{U})}{\holim}\, \, F.$$
\end{defin}

It can be shown that $T_{\vec{\jmath}}F$ is a polynomial functor of degree $\leq\vec{\jmath}$ which satisfies certain universal properties \cite[Theorem 4.2]{MV:Multi} and can be thought of as a $\vec{\jmath}\;^{th}$ multivariable polynomial approximation to $F$.  There are canonical maps $T_{\vec{\jmath}}F\to T_{\vec{k}}F$ for $\vec{k}<\vec{\jmath}$ and $F\to T_{\vec{\jmath}}F$ which combine into the \emph{Taylor multi-tower for $F$}.  We will denote by $T_{\vec{\infty}}F$ be the homotopy inverse limit of this multi-tower.

We next recall an important result about the convergence of the multivariable Taylor tower for embeddings. 

\begin{thm}[\cite{MV:Multi}, Theorem 4.16]\label{T:MultiConvergence}
Let $M=\coprod_{i=1}^{m}P_i$, $\vec{P}=(P_1,\ldots, P_m)$, where $\dim(P_i)=p_i$ and suppose $n-p_i-2>0$ for all $i$.  Then the map 
$$\Emb(\vec{P},N)\longrightarrow T_{\vec{\infty}}\Emb(\vec{P},N)$$
is an equivalence.
\end{thm}

The single variable version, from which \refT{MultiConvergence} follows fairly easily, was proved in \cite{GK}. We will encounter an application of these results to convergence of certain cosimplicial spaces in \refT{KnotsCosimplicialModel}.

The examples we will focus on exclusively in this paper are built out of those in \refD{someexampleswithboundary} by letting $P_i$ be the unit interval $I$ for all $i$.  More precisely, we have

\begin{defin}\label{D:linkspaces}
Let $N$ be a manifold with boundary. Choose a basepoint in $\Emb_\del(\vec{I},N)$. Define the space of \emph{string links} of $m$ strands in $N$ to be
\begin{equation}\label{E:LinkHofiber}
\LK_m(\vec{I}, N)=\hofiber \Big(\Emb_\del(\vec{I}, N)\hookrightarrow \Imm_\del(\vec{I}, N)\Big)
\end{equation}
Define the space of \emph{homotopy string links} of $m$ strands in $N$ to be
\begin{equation}\label{E:HolinkHofiber}
\HLK_m(\vec{I}, N)=\Link_\del(\vec{I},N)
\end{equation}
\end{defin}

\begin{rems}\label{R:LinkDefs}\ \ 

\begin{itemize}
\item When $N=I^n$, the basepoint for the homotopy fiber in \refE{LinkHofiber} will be the standard unlink in $I^n$, i.e.~a neat embedding of $\coprod_{i=1}^m I$ of constant unit-length derivative $\vec{e}$ which sends the endpoints  of the $m$ copies of $I$ to some fixed points $y_0^i\in I^{n-1}\times\{0\}$ and $y_1^i\in I^{n-1}\times\{1\}$, $1\leq i\leq m$. 
\item The reason we take the this version of link spaces rather than ordinary ones (embeddings or link maps of intervals rather than circles) is that cosimplicial models from Section \ref{S:CosimplicialModel} work out more conveniently in this case.  The theory of string links in $I^n$ is closely related to the theory of based links in $S^n$;
\item If $m=1$, we write $\K(I,N)$ in place of $\LK_1(I,N)$. This is the space of long knots in $N$;
\item  If $n=3$, then $\LK_m(\vec{I},I^3)$ is on $\pi_0$ the space of framed links whose framing number is even.  This is a straightforward generalization of Proposition 5.14 from \cite{V:FTK}.
\end{itemize}
\end{rems}

The passage to the homotopy fiber over the space of immersions in \refE{LinkHofiber} is simply a convenience.  As will be discussed later (see Remark \ref{R:NoTangents}), the upshot of this is that one can safely forget about certain tangential data.    In the case $N=I^n$, the difference between $\Emb_\del(\vec{I}, I^n)$ and $\LK_m(\vec{I}, I^n)$ is simple to describe.

\begin{prop}
The inclusion from equation (\ref{E:LinkHofiber}) is null-homotopic and there is an equivalence
$$\LK_m(\vec{I}, I^n)\simeq  \Emb_\del(\vec{I}, I^n)\times \prod_m \Omega^2 S^{n-1}.
$$
\end{prop}

\begin{proof}
The case $m=1$ is given in \cite[Proposition 5.17]{S:OKS} and this argument is a straightforward generalization of it. First note there is an equivalence 
$$
\Imm_\del(\vec{I}, I^n)\simeq\prod_m \Imm_\del(I,I^n)
$$ 
since an immersion of a disjoint union is just an immersion of each component.

Recall also that the map $\Imm(I,I^n)\rightarrow\Omega S^{n-1}$ which sends an immersion $g$ to the unit derivative $u(g')$ is an equivalence by the Smale-Hirsh Theorem.

Given $\vec{f}=(f_1,\ldots, f_m)$, consider the map

$$\rho(\vec{f})\colon \Delta^2 \longrightarrow \prod_{m} S^{n-1}$$

which sends $(t_1,t_2)$ to either $u(\vec{f}(t_1)-\vec{f}(t_2))$ if $t_1\neq t_2$, or $u(\vec{f}'(t))$ if $t_1=t_2=t$ (here $u((\vec{v}_1,\ldots, \vec{v}_m))=(\vec{v}_1/|\vec{v}_1|,\ldots, \vec{v}_m/|\vec{v}_m|)$). The map $\rho(\vec{f})$ may be viewed as a homotopy between $ev_1(\vec{f})$ (the restriction to the $t_1=t_2$ edge) and the restriction to the $t_1=0$ and $t_2=1$ edges. The restriction to the latter two edges is canonically null-homotopic since the images lie in the product of the southern hemispheres of the product of spheres, which is contractible. Hence $ev_1$ is null-homotopic, and so the inclusion of the embedding space is null-homotopic.

Therefore,
$$\LK_m(\vec{I},N)=\hofiber \Big(\Emb_\del(\vec{I}, I^n)\hookrightarrow \Imm_\del(\vec{I}, I^n)\Big)\simeq  \Emb_\del(\vec{I}, I^n)\times \Omega \prod_m \Omega S^{n-1},$$
and, since $\Omega$ commutes with products, the result follows.
\end{proof}

 In the special case of link spaces, the definition of the stages $T_{\vec{\jmath}}F$ of the Taylor multi-tower reduces to the homotopy limit of a finite diagram as follows.  
It is this finite model that we will use in the rest of the paper to construct cosimplicial models for link spaces.


\begin{prop}[\cite{MV:Multi}, Proposition 7.1]\label{P:FiniteTower}
Let $F=\LK_m(\vec{I},N)$ or $\HLK_m(\vec{I},N)$. Let $\vec{\jmath}=(j_1,\ldots, j_m)\geq\vec{-1}$ be a tuple of nonnegative integers, and $A^i_0,\ldots A_{j_i}^i$ be pairwise disjoint nonempty closed connected subintervals of the interior of $I$ for every $i$. Then
$$
T_{\vec{\jmath}}F(\vec{I})\simeq \underset{\vec{S}\neq\vec{\emptyset}}{\holim}\, \, F(\vec{I}-\vec{A}_{\vec{S}}).
$$
\end{prop}

We will often refer to the stages $T_{\vec{\jmath}}\LK_m(\vec{I},N)$ and $T_{\vec{\jmath}}\HLK_m(\vec{I},N)$ as homotopy limits of diagrams of ``punctured links".  Note that, because of Theorem  \ref{T:MultiConvergence}, the Taylor tower for $\LK_m(\vec{I},N)$  converges to this space when $n>3$, but we do not have an analogous result for $\HLK_m(\vec{I},N)$, even when $N=I^n$.

As mentioned in Remark \ref{R:LinkDefs}, when $m=1$ and $N=I^n$, we obtain, in the link case, the standard space of (long) knots $\K(I,I^n)$.  For $n=3$, we get classical long knot theory where $\pi_0(\K(I,I^3))$ (isotopy classes) and $H^0(\K(I,I^3))$ (invariants) are of most concern.  Some of what the authors plan to do following the foundational work laid out here will be generalizations of what has been done so far in the case of knots to the case of links.  More details about future work can be found in the Introduction.


\section{(Multi-)cosimplicial spaces and cubical diagrams}\label{S:CosimplicialReview}



\subsection{Cosimplicial spaces}\label{S:CosimplicialSpaces}


Before we introduce cosimplicial models for our link spaces, in this section we give a brief overview of cosimplicial spaces.  The importance of such objects is that they come equipped with spectral sequences converging (under favorable circumstances) to the homology and homotopy groups of their homotopy limits (or \emph{totalizations}; see \refD{Tot} below).  Further, any diagram of spaces of spectra can be turned into a cosimplicial diagram whose totalization is weakly equivalent to the homotopy limit of the original diagram. We review some of these ideas here.  Further details can be found in \cite{BK, B:HSS, Ship:Conv}.

Let $\Delta$ denote the category whose objects are finite ordered sets $[p]$, $p\geq 0$, and whose morphisms are order-preserving maps.  Such maps are generated by injections
\begin{align*}
d^i\colon [p] & \longrightarrow [p+1], \ \ \ \ \ \ \ \ \ \ \ \ \ \ \ \ \ 0\leq i\leq p+1 \\
k  & \longmapsto  \begin{cases} k, &  k<i \\  k+1, &  k\geq i\end{cases}
\end{align*}
and surjections
\begin{align*}
s^i\colon [p+1] & \longrightarrow [p], \ \ \ \ \ \ \ \ \ \ \ \ \ 0\leq i\leq p \\
k  & \longmapsto  \begin{cases} k, &  k\leq i \\  k-1, &  k> i\end{cases}
\end{align*}
which satisfy the relations

\begin{align}\label{E:Identities}
d^jd^i & =d^id^{j-1}, \ \ \ \ \ \ \ \ \ \ i<j; \notag \\
s^jd^i & =
\begin{cases}
d^is^{j-1}, & \ \ \ \ \ i<j; \\
Id, & \ \ \ \ \ i<j,\  i=j+1;  \\
d^{i-1}s^{j}, & \ \ \ \ \ i>j+1;
\end{cases}  \\
s^js^i & =s^{i-1}s^{j}, \ \ \ \ \ \ \ \ \ \ i>j. \notag
\end{align}

\begin{defin}  For $\mathcal{C}$ a category, a covariant functor $X\colon \Delta\to \mathcal{C}$ is called a \emph{cosimplicial object} in $\mathcal{C}$. It is customary to write such a functor using the notation $X^\bullet$, and its value on $[p]$ as $X^{[p]}$. The morphisms $X(d^i)$ and $X(s^i)$ are called \emph{cofaces} and \emph{codegeneracies}, respectively, while the relations \eqref{E:Identities} are called \emph{cosimplicial identities}. By abuse of notation, we will suppress the functor $X$ on morphisms and write $d^j$ (resp.~$s^j$) in place of $X(d^j)$ (resp.~$X(s^j)$).
\end{defin}
In this paper, we will only be concerned with cosimplicial spaces, i.e.~functors from $\Delta$ to the category of topological spaces.

\begin{example}
The most basic example of a cosimplicial space is the \emph{cosimplicial simplex} $\Delta^{\bullet}$ which has the standard $p$-simplex $\Delta^p$ as its $p^{th}$ space.  The cofaces and codegeneracies are defined by inclusions and projections of faces.
\end{example}

\begin{defin}\label{D:Tot}  For a cosimplicial space $\Xdot$, define its \emph{totalization} $\Tot\Xdot$ to be the space of cosimplicial maps from $\Delta^{\bullet}$ to $\Xdot$, i.e.~the space of sequences of maps 
$$
f_p\colon \Delta^p\longrightarrow X^{[p]}, \ \ \ p\geq0,
$$
such that the diagrams 
$$
\xymatrix{\Delta^{p} \ar[r]^{d^{j}} \ar[d]_{f_p} & \Delta^{p+1}\ar[d]^{f_{p+1}}
\\
X^{[p]} \ar[r]^{d^{j}} & X^{[p+1]}
}\ \ \ \ \
\mbox{and} \ \ \ \ \
\xymatrix{\Delta^{p+1} \ar[r]^{s^{j}} \ar[d]_{f_{p+1}} & \Delta^{p}\ar[d]^{f_p}
\\
X^{[p+1]} \ar[r]^{s^{j}} & X^{[p]}
}
$$
\end{defin}
commute for all cofaces $d^j$ and codegeneracies $s^j$.

Note that one can also consider \emph{partial totalizations} $\Tot^j\Xdot$ which consist of cosimplicial maps between $j^{th}$ truncations of $\Delta^{\bullet}$ and $\Xdot$ (cosimplicial maps $f_i:\Delta^i\to X^i$, $i\leq j$).  In this case, totalization $\Tot\Xdot$ can be defined as
\begin{equation}\label{E:TotTower}
\Tot \Xdot=
\lim_{\longleftarrow}(\Tot^{0}\Xdot \longleftarrow
\Tot^{1}\Xdot\longleftarrow \cdots).
\end{equation}
It is not hard to see that the maps $\Tot^j\Xdot\longleftarrow\Tot^{j+1}\Xdot$ are fibrations.

In practice, we will often use an alternative definition of totalization which is that $\Tot^j X^{\bullet}$ is the homotopy limit of the 
restriction of $X^{\bullet}$ to the full subcategory $\Delta^{\leq j}$ consisting of 
objects $[0], ..., [j]$ \cite[Ch. XI, \S4]{BK}. Another equivalent definition is as the homotopy limit of the 
pullback to the poset of nonempty subsets of $[j]$ (see, for example, \cite[Section 2.1]{S:OKS}).   For $X^{\bullet}$ fibrant, the definitions as spaces of cosimplicial maps and homotopy limits of are the same.  So from now on, if $X^{\bullet}$ is not fibrant, it will be understood that we are using the homotopy limit definition or that, equivalently, we are using \refD{Tot} but with the unstated assumption that it is being applied to a fibrant replacement of $X^{\bullet}$.


\subsection{Cubical diagrams}\label{S:Cubes}


The purpose of this section is to remind the reader of the relevant definitions and a few useful results about cubical diagrams and their homotopy limits. We will explore the relationship between cubical diagrams and cosimplicial spaces in \refS{CosimplicialCubical}. Details can be found in \cite[Section 1]{CalcII}.

\begin{defin}
Let $T$ be a finite set. A \textsl{$\abs{T}$-cube} of spaces is a (covariant) functor 
$$\mathcal{X}:\mathcal{P}(T)\longrightarrow\Top.$$
\end{defin}

\begin{defin}
The \textsl{total homotopy fiber}, or \textsl{total fiber}, of a $\abs{T}$-cube $\mathcal{X}$ of based spaces, denoted $\tfiber(S\mapsto \mathcal{X}(S))$ or $\tfiber(\mathcal{X})$, is the homotopy fiber of the map 
$$
\mathcal{X}(\emptyset)\longrightarrow\underset{S\neq\emptyset}{\holim}\, \mathcal{X}(S).
$$ 
If this map is a $k$-connected, we say the cube is \emph{$k$-cartesian}. In case $k=\infty$, i.e., the map is a weak equivalence, we say the cube $\mathcal{X}$ is \textsl{homotopy cartesian}.
\end{defin}

Note we can also talk about the total homotopy fiber of a cube of unbased spaces, but this requires a choice of basepoint in the homotopy limit.

The total fiber can also be thought of as an iterated homotopy fiber. That is, view a $\abs{T}$-cube $\mathcal{X}$ as a map of $(\abs{T}-1)$-cubes $\mathcal{Y}\rightarrow\mathcal{Z}$. In this case, 
$$
\tfiber(\mathcal{X})=\hofiber(\tfiber(\mathcal{Y})\rightarrow\tfiber(\mathcal{Z})).
$$ 
This description will be used in \refL{switcharrows}. More precisely,

\begin{prop}\label{P:IteratedHofiber}
Suppose $\mathcal{X}$ is a $\abs{T}$-cube of based spaces, and let $t\in T$. Then $$\tfiber(\mathcal{X})\simeq\hofiber(\tfiber(S\mapsto \mathcal{X}(S))\longrightarrow\tfiber(S\mapsto \mathcal{X}(S\cup\{t\})),$$ 
where $S$ ranges through subsets of $T-\{t\}$.
\end{prop}

One easy but useful consequence of this is the following.  This result will be used in the proof of \refP{multiembcartesian}.

\begin{prop}[\cite{MV:Multi}, Proposition 3.9]\label{P:juxtaposecubes}
Let $\mathcal{X}$ and $\mathcal{Y}$ be $k$-cartesian $\abs{T}$-cubes of based spaces, and suppose that, for some $t\in T$, $\mathcal{X}(S\cup\{t\})=\mathcal{Y}(S)$ for all $S\subset T-\{t\}$. Then the $\abs{T}$-cube $\mathcal{Z}$ defined by $\mathcal{Z}(S)=\mathcal{X}(S)$ and $\mathcal{Z}(S\cup\{t\})=\mathcal{Y}(S\cup\{t\})$ for $S\subset T-\{t\}$ is also $k$-cartesian.
\end{prop}



\begin{defin}
Let $\mathcal{P}_0(T)$ denote the subposet of nonempty subsets of $T$. We call the restriction of $\mathcal{X}$ to this subposet a \emph{subcubical diagram}, or a sub-$\abs{T}$-cube. 
\end{defin}

The homotopy limit $\holim_{\mathcal{P}_0(T)} \mathcal{X}$, often denoted $\holim_{S\neq \emptyset} \mathcal{X}(S)$ (as in the definition of total homotopy fiber), can in special situations be described as a partial totalization of a cosimplicial space (see Section \ref{S:CosimplicialReview}).



Suppose $S$ and $T$ are finite sets. A $\abs{S\coprod T}$-cube of spaces can be viewed as a $\abs{S}$-cube of $\abs{T}$-cubes (and vice versa). This is simply a reflection of the categorical equivalence $\mathcal{P}(S\coprod T)\cong \mathcal{P}(S)\times\mathcal{P}(T)$. We may also view a functor $\mathcal{X}:\mathcal{P}_0(S)\times \mathcal{P}_0(T)\rightarrow\Top$ as a sub-$\abs{S}$-cube of sub-$\abs{T}$-cubes. Moreover, we have a homeomorphism
\begin{equation}\label{E:HolimProduct}
\underset{\mathcal{P}_0(S)\times\mathcal{P}_0(T)}{\holim}\, \mathcal{X}\cong\underset{\mathcal{P}_0(S)}{\holim}\, \underset{\mathcal{P}_0(T)}{\holim}\, \mathcal{X}
\end{equation}
which is a general fact about homotopy limits; see for example \cite[Ch. XI, Example 4.3]{BK}.
(The order of $S$ and $T$ does not matter of course, since homotopy limits commute.)

\vskip 6pt
We next wish to establish \refP{multiembcartesian}, which is a result about cubical diagrams of embeddings and which will be used in the proof of \refP{HomotopyConvergence}.
To do this, we will first need the following two statements. The first a well-known consequence of the higher Blakers-Massey Theorem, and we include the proof for the reader who is familiar with this theorem. Details can be found in \cite{CalcII}.

Let $M$ be a smooth compact manifold and fix an embedding $e:M\rightarrow N$, where $N$ is connected. Let $\{x_1,\ldots, x_k\}$ be a set of distinct elements of $M$.

\begin{prop}\label{P:embcartesian}
The $k$-cube $S\mapsto \Emb(S,N)$, where $S\subset \{x_1,\ldots x_k\}$, is $((k-1)(n-2)+1)$-Cartesian. 
\end{prop}

\begin{proof}
Let $T\subset\{x_1,\ldots, x_{n-1}\}$, and consider the $(k-1)$-cube 
$$
T\longmapsto \hofiber\left(\Emb(T\cup\{x_n\},N)\rightarrow\Emb(T,N)\right).
$$
Clearly 
$$
\hofiber\left(\Emb(T\cup\{x_n\},N)\rightarrow\Emb(T,N)\right)\simeq N-T.
$$
Moreover, it is easy to verify that the cube $T\mapsto N-T$ is strongly cocartesian, and that the maps 
$$
N-\{x_1,\ldots, x_{n-1}\}\longrightarrow N-\{x_1,\ldots, \widehat{x_i},\ldots, x_{n-1}\}
$$ 
are $(n-1)$-connected for every $i$. Hence by Theorem 2.3 of \cite{CalcII}, the cube $T\mapsto N-T$ is $(k-1)(n-1)+1-k=((k-1)(n-2)+1)$-cartesian.
\end{proof}

\begin{rem}\label{R:GeneralCartesian}
More generally, suppose $T\subset \{x_1,\ldots, x_k\}$ is a nonempty subset with complement $C=\{x_1,\ldots, x_k\}-T$. Then the same argument as used in the above proof shows that the $\abs{T}$-cube $S\mapsto\Emb(S\cup C,N)$ is $((\abs{T}-1)(n-2)+1)$-cartesian.
\end{rem}

Let $k,m$ be a positive integers, and let $T=\coprod_{i=1}^k T_i$ be a set of distinct points in $M$ such that  $\abs{T_i}=m$ for all $i$. For a subset $S\subset\{1,\ldots, k\}$, let $T_S=T-\cup_{i\in S} T_i$.

\begin{prop}\label{P:multiembcartesian}
The $k$-cube $S\mapsto \Emb(T_S,N)$, where $S\subset \{1,\ldots k\}$, is $((k-1)(n-2)+1)$-cartesian.
\end{prop}

\begin{proof}
This will follow from the previous proposition as well as Proposition \ref{P:juxtaposecubes}. Let $K=(k-1)(n-2)+1$. For all subsets $U\subset T$ and $C\subset T-U$  with $\abs{U}=k$, by Remark \ref{R:GeneralCartesian}, the $k$-cubes $R\mapsto\Emb(R\cup C)$ for $R\subset U$ are $K$-cartesian. Repeated application of Proposition \ref{P:juxtaposecubes} allows us to conclude that $S\mapsto \Emb(T_S,N)$ is $K$-cartesian.
\end{proof}

As an illustration of Proposition \ref{P:multiembcartesian}, let us consider the case $k=m=2$. We have the following diagram:

$$\xymatrix{
\Emb(x_1,x_2,y_1,y_2)  \ar[r]\ar[d] & \Emb(x_2,y_1,y_2) \ar[r]\ar[d] & \Emb(x_2,y_2)\ar[d] \\
\Emb(x_1,y_1,y_2)  \ar[r]\ar[d] & \Emb(y_1,y_2) \ar[r]\ar[d] & \Emb(y_2)\ar[d] \\
\Emb(x_1,y_1)  \ar[r] & \Emb(y_1) \ar[r] & \Emb(\emptyset) \\
}
$$

The remark following Proposition \ref{P:embcartesian} shows that each of the four squares in this diagram is $(n-1)$-cartesian. Repeated application of Proposition \ref{P:juxtaposecubes} allows us to conclude that 

$$\xymatrix{
\Emb(x_1,x_2,y_1,y_2)  \ar[r]\ar[d] & \Emb(x_2,y_2)\ar[d] \\
\Emb(x_1,y_1)  \ar[r] & \Emb(\emptyset) \\
}
$$

is also $(n-1)$-cartesian.


\subsection{Relationship between cosimplicial spaces and cubical diagrams}\label{S:CosimplicialCubical}



Since we wish to convert the diagrams defining stages of the Taylor multi-towers for spaces of links into cosimplicial diagrams, we establish here some of the dictionary for doing so.  Most of the definitions and results we need have been collected in \cite{S:TSK} and we will often quote directly from that paper.  We also provide the proofs of some statements we could not find in the literature.   The main result that we need from this section is \refP{CubeConvergence} (this is Theorem 7.3 in \cite{S:TSK}),  which will give us an alternative way of checking the convergence of the Bousfield-Kan homotopy spectral sequence of a cosimplicial space.

\begin{defin}[\cite{S:TSK}, Definition 6.3]
Let $c_j:\mathcal{P}_0([j])\rightarrow\Delta$ be the (covariant) functor defined by $S\mapsto[\abs{S}-1]$, and which sends $S\subset S'$ to the composite $[\abs{S}-1]\cong S\subset S'\cong[\abs{S'}-1]$, where the two isomorphisms are isomorphisms of ordered sets. 
\end{defin}

For instance, let $j=2$, $S=\{1,2\}$, and $S'=\{0,1,2\}$. Then $S\cong [1]$ by $1\mapsto 0$ and $2\mapsto 1$, and $S'\cong [2]$ by $i\mapsto i$ for $i=0,1,2$. Hence the map $[1]\rightarrow[2]$ induced by the inclusion $S\subset S'$ is the map $0\mapsto 1$ and $1\mapsto 2$. Recall that for a cosimplicial space $X^\bullet$, we write $X^{[p]}$ for the value of $X^\bullet$ on $[p]$. Then $X\circ c_2$ gives rise to the subcubical diagram

$$
\xymatrix@=20pt{
         &           &   X^{[0]} \ar'[d][dd]           \ar[dr]  &                  \\
        &  X^{[0]} \ar[rr] \ar[dd]  &             & X^{[1]}
         \ar[dd] \\
X^{[0]} \ar'[r][rr] \ar[dr] &        &   X^{[1]}
\ar[dr] &                   \\
        &   X^{[1]} \ar[rr]      &                    &
X^{[2]}
}
$$

The maps in this commutative diagram are the coface maps associated with the various inclusion maps as above. Note that the codegeneracies play no role in this subcubical diagram.  The following is \cite[Theorem 6.6]{S:TSK} (see also \cite[Lemma 2.9]{S:OKS}).

\begin{lemma}\label{L:totandholim}
For a (fibrant) cosimplicial space $X^\bullet$, there is an equivalence
 $$\Tot^j(X^\bullet)\simeq\underset{\mathcal{P}_0([j])}{\holim}\, X\circ c_j.$$
\end{lemma}

Assume a basepoint in $\Tot^{j-1}X^\bullet$ has been chosen.

\begin{defin}
Let $L_jX^\bullet=\hofiber(\Tot^jX^\bullet\rightarrow\Tot^{j-1}X^\bullet)$.
\end{defin}

For the following lemma it will be convenient to write the homotopy limits in a way that indicates functoriality more explicitly for reasons that will become apparent shortly. Thus we will sometimes write $\underset{c\in\mathcal{C}}{\holim}\, F(c)$ in place of $\underset{\mathcal{C}}{\holim}\, F$.

\begin{lemma}\label{L:totlayers}  There is an equivalence
$$L_jX^\bullet\simeq\hofiber\left(X^{[0]}\longrightarrow  \underset{S\in\mathcal{P}_0([j-1])}{\holim}\, X\circ c_{j-1}(S\cup\{j\})\right).$$ 
That is, $L_jX^\bullet$ is the total homotopy fiber of a $j$-cube of spaces.
\end{lemma}

\begin{proof}
By Lemma \ref{L:totandholim}, 
$$
L_jX^\bullet=\hofiber\left(\underset{T\in\mathcal{P}_0([j])}{\holim}\, X\circ c_j(T) \rightarrow \underset{S\in\mathcal{P}_0([j-1])}{\holim}\, X\circ c_{j-1}(S)\right).
$$ 
By inspection, 
$$\underset{T\in\mathcal{P}_0([j])}{\holim}\, X\circ c_j(T)\cong\holim\left(X^{[0]}\rightarrow \underset{S\in\mathcal{P}_0([j-1])}{\holim}\, X\circ c_{j-1}(S)\leftarrow \underset{S\in\mathcal{P}_0([j-1])}{\holim}\, X\circ c_{j-1}(S\cup\{j\})\right).$$
Fibering this over $\underset{\mathcal{P}_0([j-1])}{\holim}\, X\circ c_{j-1}(S)$ yields
$$
L_jX^\bullet\simeq\hofiber\left(X^{[0]}\rightarrow \underset{\mathcal{P}_0([j-1])}{\holim}\, X\circ c_{j-1}(S\cup\{j\})\right).
$$
\end{proof}

\begin{example}
When $j=2$, this lemma says that $L_2X^\bullet$ is the total homotopy fiber of the following square diagram.
$$
\xymatrix{
X^{[0]}\ar[d]\ar[r] & X^{[1]}\ar[d]\\
X^{[1]}\ar[r] & X^{[2]}
}
$$
\end{example}

\begin{defin}[\cite{S:TSK}, Definition 7.2]
Let $c^{!}_j:\mathcal{P}(\underline{j})\rightarrow\Delta$ be the functor defined by $c^{!}_j(S)=[j-\abs{S}]$, and which associates to an inclusion $S\subset S'$ the map 
$$
[j-\abs{S}]\cong[j]-S\longrightarrow[j]-S'\cong[j-\abs{S'}].
$$  Here the middle map sends $i\in j-\abs{S}$ to the largest element of $j-\abs{S'}$ less than or equal to $i$.
\end{defin}


For example, suppose $j=3$, $S=\{2\}$, and $S'=\{2,3\}$. Then $[2]\cong[3]-\{2\}=\{0,1,3\}$ by the map $0\mapsto 0$, $1\mapsto 1$, and $2\mapsto 3$. Then the map $\{0,1,3\}\to \{0,1\}=[3]-\{2,3\}$ sends $0$ to $0$, $1$ to $1$, and $3$ to $1$. Finally, $[3]-\{2,3\}\cong[1]$ by the identity map. Putting this all together, the map $[2]\rightarrow[1]$ induced by the above inclusion sends $0$ to $0$, $1$ to $1$, and $2$ to $1$. 

Although the codegeneracies did not play a role previously, they will play an important one in the following proposition, which will be used in the proof of Theorem \ref{T:MainThmLinks}.

\begin{prop}\label{P:CubeConvergence}
For any (fibrant) cosimplicial space $X^{\bullet}$,  there is an equivalence 
$$L_jX^{\bullet}\simeq\Omega^j \tfiber(X^\bullet\circ c^{!}_j).$$
\end{prop}

This will follow from the following two lemmas. The first is well-known; the second is quoted but not proved in \cite{S:TSK}, and we include a proof which we learned from Goodwillie.

\begin{lemma}\label{L:hofiberreverse}
Let $X$ and $Y$ be based spaces, and suppose there are maps $X\rightarrow Y$ and $Y\rightarrow X$ such that the composition $X\rightarrow Y\rightarrow X$ is the identity. Then there is an equivalence 
$$
\hofiber(X\rightarrow Y)\rightarrow\Omega\hofiber(Y\rightarrow X).
$$
\end{lemma}

\begin{proof}
Consider the square $\mathcal{A}$:

$$
\xymatrix{
X\ar[d]\ar[r] & Y\ar[d]\\
X\ar[r] & X
}
$$

We can interpret the total homotopy fiber of this square as iterating horizontal fibers or iterating vertical fibers. These respectively yield equivalences
$$
\hofiber(X\rightarrow Y)\simeq\tfiber(\mathcal{A})\simeq\Omega\hofiber(Y\rightarrow X).
$$
\end{proof}

\begin{lemma}\label{L:switcharrows}
Suppose that $\mathcal{X}:\mathcal{P}(\underline{j})\rightarrow\Spaces_\ast$ is a $j$-cube of based spaces and that for each $S\subset\underline{j}$ and $t\in\underline{j}-S$, there are maps $\mathcal{X}(S\cup\{t\})\rightarrow\mathcal{X}(S)$ such that the composed maps $\mathcal{X}(S)\rightarrow\mathcal{X}(S\cup\{t\})\rightarrow\mathcal{X}(S)$ are the identity. Then there is an equivalence 
$$\tfiber(S\longmapsto\mathcal{X}(S))\simeq\Omega^j\tfiber(S\longmapsto\mathcal{X}(\underline{j}-S)).$$
\end{lemma}

\begin{proof}
This will follow from repeated application of Lemma \ref{L:hofiberreverse}. Let $s_1\in\underline{j}$ be arbitrary, and let $T$ range through subsets of $\underline{j}-\{s_1\}$. Using the interpretation of total homotopy fiber as an iterated homotopy fiber,
$$\tfiber(S\longmapsto\mathcal{X}(S))\simeq\tfiber(T\longmapsto\hofiber(\mathcal{X}(T)\rightarrow\mathcal{X}(T\cup\{s_1\})))$$

By Lemma \ref{L:hofiberreverse}, 
$$\tfiber(T\mapsto\hofiber(\mathcal{X}(T)\rightarrow\mathcal{X}(T\cup\{s_1\})))\simeq\Omega\tfiber(T\mapsto\hofiber(\mathcal{X}(T\cup\{s_1\})\rightarrow\mathcal{X}(T))).$$

Hence $\tfiber(S\mapsto\mathcal{X}(S))\simeq\Omega\tfiber(S\mapsto\mathcal{Y}(S))$, where $\mathcal{Y}(T\cup\{s_1\})=\mathcal{X}(T)$, and $\mathcal{Y}(T)=\mathcal{X}(T\cup\{s_1\})$ for $T\subset S-\{s_1\}$. Now we can apply the same argument to $\mathcal{Y}$ using $s_2\neq s_1$, and continue until all $s\in \underline{j}$ have been exhausted.
\end{proof}

\begin{proof}[Proof of \refP{CubeConvergence}]
By \refL{totandholim}, we have an equivalence,

$$L_jX^\bullet\simeq\hofiber\left(X^{[0]}\longrightarrow  \underset{S\in\mathcal{P}_0([j-1])}{\holim}\, X\circ c_{j-1}(S\cup\{j\})\right).$$ 

We may write homotopy fiber on the right side of this equivalence as the total homotopy fiber of a $j$-cube $\tfiber(S\mapsto X\circ c_{j-1}(S))$. Then by \refL{switcharrows} we obtain an equivalence

$$\tfiber(S\longmapsto X\circ c_{j-1}(S))\simeq\Omega^j\tfiber(S\longmapsto X\circ c_j(\underline{j}-S)).$$

We leave it to the reader to check, using the cosimplicial identities, that the hypotheses of Lemma \ref{L:switcharrows} are satisfied. 

\end{proof}

\begin{example}
As an illustration of Proposition \ref{P:CubeConvergence}, let us consider the case $j=2$. The square diagram we obtain by composing $X^\bullet$ with $c_2^!$ is
$$
\xymatrix{
X^{[2]}\ar[d]\ar[r] & X^{[1]}\ar[d]\\
X^{[1]}\ar[r] & X^{[0]}
}
$$
The arrows in this diagram are the codegeneracies. The claim asserts that the total fiber of this square is equivalent to $L_2X^{\bullet}$. Now consider the diagram
$$
\xymatrix{
X^{[0]}  \ar[r]\ar[d] & X^{[1]} \ar[r]\ar[d] & X^{[0]}\ar[d] \\
X^{[1]}  \ar[r]\ar[d] & X^{[2]} \ar[r]\ar[d] & X^{[1]}\ar[d] \\
X^{[0]}  \ar[r] & X^{[1]} \ar[r] & X^{[0]} \\
}
$$
Proceeding as in Lemma \ref{L:switcharrows}, we see that the total fiber of the upper left square is equivalent to the loopspace of the total fiber of the upper right square. Similarly, the total fiber of the upper right square is equivalent to the loopspace of the total fiber of the lower right square. Combining these, we achieve the desired result. We could have equivalently gone through the lower left square to deduce this. Once again, this involves using the cosimplicial identities to check that the composed arrows really are the identity.
\end{example}


\subsection{Bousfield-Kan spectral sequences for a cosimplicial space}\label{S:B-KSSs}


If one is interested in computing the (co)homology or homotopy groups of $\Tot\Xdot$, one can use the \emph{Bousfield-Kan spectral sequences} associated to $\Xdot$, defined as follows.

The (second-quadrant) homotopy spectral sequence for $\Xdot$ is simply the Eilenberg-Moore  spectral sequence associated to the tower of fibrations
$$
\Tot^{0}\Xdot \longleftarrow
\Tot^{1}\Xdot\longleftarrow \cdots.
$$
Its first page is given by \cite[Ch. X, \S6]{BK}
\begin{equation}\label{E:FirstPageHomotopy}
E^{1}_{-p,q}= \pi_q(X^{[p]})\bigcap\Big(\bigcap_{i=0}^{p-1}ker(s^{i}_{*})\Big)
\end{equation}
where $s^{i}_{*}\colon \pi_q(X^{[p+1]})\to \pi_q(X^{[p]})$ is the map induced on $q$th homotopy group by the codegeneracy $s^i$.  

One could take as the $E^1$ term the usual homotopy groups of $X^{[p]}$, in which case this page is just the chain complex of the cosimplicial abelian group $\pi_*(\Xdot)$ (if each $X^{[p]}$ is connected).  Intersecting with the kernel of the codegeneracies gives a more efficient way of constructing $E^1$ (this is known as the \emph{normalization} of $\Xdot$).  Whether $\Xdot$ is normalized or not, the $E^2$ page  comes out to be the same since $\pi_*(\Xdot)$ and its normalized version are quasi-isomorphic.

The first differential is given by the alternating sum of the maps induced on homotopy groups by cofaces.  More precisely, we have
\begin{align*}
d_1 \colon  E^{1}_{-p,q} & \longrightarrow E^{1}_{-p-1,q} \\
                   \alpha & \longmapsto \sum_{i=0}^{p+1}(-1)^i d^{i}_{*}(\alpha).
\end{align*}

An important question is that of the convergence of the homotopy spectral sequence.  One situation under which this spectral sequence in fact converges to $\pi_*(\Tot\Xdot)$ is the following, which follows from \cite[Ch. IX, Lemma 5.6]{BK}.

\begin{prop}\label{P:HomotopyConvergenceCondition} The homotopy spectral sequence of a tower of fibrations
$Y_1\leftarrow Y_2\leftarrow\cdots$ converges to $\underset{\leftarrow}{\lim}\;Y_j$ if the connectivity of the homotopy fibers of the maps $Y_j\to Y_{j-1}$ increases with $j$.
\end{prop}

Thus the spectral sequence from \eqref{E:FirstPageHomotopy} converges to $\Tot\Xdot$ if the connectivity of 
$$L_jX^{\bullet}=\hofiber(\Tot^{j}\Xdot \to \Tot^{j-1}\Xdot)$$ 
grows with $j$.  Note that \refP{CubeConvergence} gives us an alternative way of checking if this is the case, provided we can compute or estimate the connectivities of the total fibers of the relevant cubical diagrams.


The (co)homology spectral sequence for $\Xdot$ (also called the generalized Eilenberg-Moore spectral sequence) is the spectral sequence associated with the double complex whose differential in the vertical direction is the usual internal differential on the (co)chains on the spaces $X^{[p]}$ and the vertical differential is obtained by taking the alternating sum of the maps induced on (co)homology by the cofaces.  We now give more details for the cohomology spectral sequence (since we are ultimately interested in $H^0$, or invariants, of the totalization of a certain cosimplicial spaces).  Dualizing in a straightforward way gives the homology version.  Standard references for this subject are \cite{B:HSS, Ship:Conv}.

The first page of the cohomology spectral sequence is given by
\begin{equation}\label{E:FirstPageCohomology}
E_{1}^{-p,q}=H^q(X^{[p]})/\big(\sum_{i=0}^{p} im((s^i)^*)\big)
\end{equation}
where $(s^{i})^{*}\colon H^q(X^{[p-1]})\to H^q(X^{[p]})$ is the map induced on $q$th cohomology group by the codegeneracy $s^i$.

One could again have taken just $H^q(X^{[p]})$ as the $E_{1}^{-p,q}$ term, but instead taking cokernels of the codegeneracies gives a normalized version of this page which is often easier to handle.

The first differential is again induced by alternating sums of cofaces, i.e.
\begin{align*}
d^1 \colon  E_{1}^{-p,q} & \longrightarrow E_{1}^{-p+1,q} \\
                   \beta & \longmapsto \sum_{i=0}^{p+1}(-1)^i (d^{i})^{*}(\beta).
\end{align*}
Here $(d^{i})^{*}$ is the map induced on cohomology by the $i^{th}$ coface.

The issue of the convergence of the cohomology spectral sequence is a bit more sensitive than that of the homotopy spectral sequence, but a sufficient condition turns out to be the following.

\begin{thm}[\cite{B:HSS}]\label{T:CohomologyConvergenceCondition}  
If each $X^{[p]}$ is connected and of finite type and if there exist a vanishing line of slope steeper than $-1$ at some page $E_i$, then the spectral sequence converges to $H^*(\Tot\Xdot)$.
\end{thm}

 Otherwise, the best that can be said is that the spectral sequence converges to $H^*(\Tot C^*\Xdot)$, i.e.~the cohomology of the simplicial abelian group obtained by applying cochains to  the spaces $X^p$.  One then has a canonical map
\begin{equation}\label{E:CochainsMap}
H^*(\Tot C^*\Xdot) \longrightarrow H^*(\Tot\Xdot)
\end{equation}
which may or may not be an isomorphism.

%

%


\subsection{$m$-cosimplicial  spaces and their spectral sequences}\label{S:m-Cosimplicial}


The generalization to $m$-cosimplicial spaces is fairly straightforward and we go through it here mostly to set the notation.  Let $(\Delta)^m=\Delta\times\Delta\times\cdots\times\Delta$ be the $m$-fold
product of the category $\Delta$ with itself.

\begin{defin}\label{D:m-cosimplicial}
An \emph{$m$-cosimplicial space} is a covariant functor $(\Delta)^m\to \Spaces$.
\end{defin}

We will denote an $m$-cosimplicial space by $X^{\vec\bullet}$ ($m$ will be understood from the context).  In particular, fixing a simplex in all but one of the factors of  $(\Delta)^m$ gives an ordinary cosimplicial space $\Xdot$ with cofaces and codegeneracies which satisfy the usual identities.  Let $\vec{p}=(p_1, p_2, ..., p_m)$, and denote by $X^{[\vec{p}\,]}$ the value of the functor $(\Delta)^m\to \Spaces$ on $[\vec{p}\,]=([p_1], [p_2], ..., [p_m])\in (\Delta)^m$.

Now let $\Delta^{\vec\bullet}$ be the $m$-cosimplicial version of $\Delta^{\bullet}$, i.e.~the $m$-cosimplicial space whose $\vec{p}=(p_1, p_2, ..., p_m)$ entry is $\Delta^{\vec{p}}=\Delta^{p_1}\times\Delta^{p_2}\times\cdots\times\Delta^{p_m}$.  
In analogy with Definition \ref{D:Tot}, we then have
\begin{defin}\label{D:mTot}
The \emph{totalization} $\Tot X^{\vec\bullet}$ of a (fibrant replacement of) $m$-cosimplicial space $X^{\vec\bullet}$ is the space of cosimplicial maps from $\Delta^{\vec\bullet}$ to $X^{\vec\bullet}$.  Alternatively, it is the homotopy limit of the $m$-cosimplicial diagram $X^{\vec{\bullet}}$.
\end{defin}

We again have partial totalizations of $X^{\vec\bullet}$.  Fixing $\vec{\jmath}=(j_1, j_2, ..., j_m)$ we define 
$$
\Tot^{\vec{\jmath}}X^{\vec\bullet}\subset \prod_{\vec{\imath}\leq \vec{\jmath}}\Map(\Delta^{\vec{\imath}}, X^{\vec{\imath}})
$$
as the subset of such maps which are compatible with all the cofaces and codegeneracies (i.e.~compatible with cosimplicial maps ``in any direction" in $X^{\vec\bullet}$).

As we will show shortly, $\Tot^{\vec{\jmath}}$ can be thought of as an iterated application of $\Tot^{j_i}$ in the $i^{th}$ variable for each $1\leq i\leq m$, and so it will be useful to have notation which says to which variable a partial totalization applies. Let $\Tot_i^k$ be the $k^{th}$ partial totalization of a multicosimplicial space $X^{\vec{\bullet}}$ in the $i^{th}$ variable.

\begin{prop}\label{P:iteratedtot}  There is an equivalence
$$\Tot^{\vec{\jmath}}X^{\vec\bullet}\simeq\Tot_1^{j_1}\Tot_2^{j_2}\cdots\Tot_m^{j_m}X^{\vec{\bullet}}.$$ Moreover, the order of the $\Tot_i^{j_i}$ does not matter.
\end{prop}

\begin{proof}
Recalling the discussion following \refD{Tot}, we think of partial totalizations as homotopy limits, i.e.
\begin{equation}\label{E:MultiTotAsHolim}
\Tot^{\vec{\jmath}}X^{\vec\bullet}\simeq\underset{\Delta^{\leq \vec{\jmath}}}{\holim}\, X^{\vec\bullet}
\end{equation}

Moreover, 
$$
\underset{\Delta^{\leq \vec{\jmath}}}{\holim}\, X^{\vec\bullet}=\underset{\Delta^{\leq j_1}\times\Delta^{\leq j_2}\times\cdots\times\Delta^{\leq j_m}}{\holim}\, X^{\vec\bullet}
$$
and we have by \eqref{E:HolimProduct}
$$
\underset{\Delta^{\leq j_1}\times\Delta^{\leq j_2}\times\cdots\times\Delta^{\leq j_m}}{\holim}\, X^{\vec\bullet}\cong\underset{\Delta^{\leq j_1}}{\holim}\,\underset{\Delta^{\leq j_2}}{\holim}\cdots\underset{\Delta^{\leq j_m}}{\holim}\, X^{\vec\bullet}.
$$
The order in which we write the homotopy limits does not matter since homotopy limits commute. Translating the homotopy limits back to partial totalizations yields the desired result.
%
\end{proof}


To get at the homology and homotopy of $\Tot X^{\vec\bullet}$,  consider the diagonal cosimplicial space
$$
X^{\vec\bullet}_{diag}=\{X^{(i,i,...,i)}\}_{i=0}^{\infty}
$$
where the cofaces and codegeneracies are compositions of cofaces and codegeneracies \emph{of the same index}; i.e.~ $s^j=(s^j, s^j,\ldots, s^j)$ and $d^j=(d^j, d^j,\ldots, d^j)$. Although there are a lot more possible morphisms from $X^{([i],[i],...,[i])}$ to $X^{([i+1],[i+1],...,[i+1])}$, it is easy to see that only these will satisfy the cosimplicial identities in $X^{\vec\bullet}_{diag}$.

Then we have the following 
\begin{prop}[\cite{Ship:Conv}, Proposition 8.1]  There is an equivalence
\begin{equation}\label{E:DiagTot}  
\Tot X^{\vec\bullet}_{diag}\simeq \Tot X^{\vec\bullet}.
\end{equation}
\end{prop}
\begin{rem}
The statement in \cite{Ship:Conv} is for bicosimplicial spaces, but the proof generalizes immediately to multicosimplicial spaces.
\end{rem}

To study the homology and homotopy of $\Tot X^{\vec\bullet}$, we can thus apply the Bousfield-Kan spectral sequences to the diagonal cosimplicial space $\Tot X^{\vec\bullet}_{diag}$.  This is precisely what we will do in Section \ref{S:LinksSS's}.


\section{Cosimplicial models for spaces of links in a Euclidean space}\label{S:CosimplicialModel}


In this section we define cosimplicial models for $\LK_m(\vec{I}, I^n)$ and $\HLK_m(\vec{I}, I^n)$ and study the Bousfield-Kan spectral sequences associated to them.  We are focusing on the case $N=I^n$ because the discussion is much simpler in this case.  However, almost everything we do holds for an arbitrary $N$, and we will remark on this where appropriate.


\subsection{Review of the cosimplicial model  for the space of long knots}\label{S:KnotsModel}


To define the multi-cosimplicial models for $\LK_m(\vec{I}, I^n)$ and $\HLK_m(\vec{I}, I^n)$ in Section \ref{S:LinksModel}, it is easiest to first review the cosimplicial model for the Taylor tower of $\K(I,I^n)$.  This model was suggested in \cite{GKW}, but Sinha \cite{S:TSK} was the first to make it precise.

The first observation is that punctured knots are homotopy equivalent to configuration spaces.  More precisely, if $A_0, ..., A_j$ are as in Section \ref{S:MultiCalc} disjoint nonempty connected subintervals of the interior of $I$, then let $A_{[j]}=\cup_{i=0}^j A_i$.  Then, recalling the definition of  $C_{\del}(j, I^n)$ from \refS{Conventions}, we have
\begin{prop}[\cite{S:TSK}, Proposition 5.15]\label{P:PuncturedKnotsConfs}
There is an equivalence
$$\K(I-A_{[j]}, I^n)\simeq C_{\del}(j, I^n).$$
\end{prop}

The homotopy equivalence in \refP{PuncturedKnotsConfs} is given by retracting each interval to, say, its midpoint.

\begin{rems}\label{R:NoTangents}
  The embeddings of punctured interval are homotopy equivalent to configurations with tangent vectors.  Since immersions are determined by their derivative, passing to the fiber over the space of immersions removes this tangential data.  Also notice that, when $j=0$, $C_{\del}(0,I^n)=\{ y_0, y_{1}\}\simeq *$, where $y_0$ and $y_1$ are the fixed points in $I^{n-1}\times\{0\}$ and $I^{n-1}\times\{1\}$ where endpoints of $I$ are mapped.  This corresponds to the situation when one hole is punched in the knot and the resulting arcs are retracted to  $y_0$ and $y_1$.
\end{rems}

In view of \refP{PuncturedKnotsConfs}, restriction of a punctured knot to a knot with one more puncture in effect introduces a new interval and hence another configuration point, so it looks like the stages of the Taylor tower $T_j\K(I,I^n)$ are homotopy limits of diagrams of configuration spaces with doubling maps.  To make this precise, we first have to compactify the configuration spaces.

Let $\phi_{ik}$ be the map from $C_{\del}(j, I^n)$ to $S^{n-1}$ given by the 
normalized difference of $i^{th}$ and $k^{th}$ configuration point.
 

\begin{defin}\label{D:DevCompactifications}
Define $C_{\del}\langle j, I^n \rangle$ to be the closure of the image of $(I^n)^{j+2}$ in 
$C_{\del}(j, I^n)\times (S^{n-1})^{{j \choose 2}}$
under $Id\times \prod_{1\leq i<k\leq j}\phi_{ik}$.
\end{defin}

This compactification was independently defined in \cite{K:OMDQ, KoSo} and in \cite{S:Compact}, although Sinha was the first to explore it in detail and show it is homotopy equivalent to $C_{\del}(j, I^n)$ \cite[Theorem 5.10]{S:Compact}.  Its main feature is that configuration points are allowed to collide while the directions of their approach are kept track of.  The space $C_{\del}\langle j, I^n \rangle$ is homeomorphic to the Fulton-MacPherson compactification of the configuration space \cite{AS,FM}, and is its  quotient  by subsets of three or more points colliding along a line.  The reason for passing to this quotient is that the genuine Fulton-MacPherson compactifications do not comprise a cosimplicial space, while spaces $C_{\del}\langle j, I^n\rangle$ do (see below).
 The stratifications of these spaces with corners have nice interpretations in terms of categories of trees \cite{S:Compact, S:OKS}.  

Now remember from  Remark \ref{R:LinkDefs} that we have a preferred unit vector $\vec{e}\in I^n$ as well as two points $y_0$ and $y_1$ in $\del N$ where a neat embedding sends the endpoints of $I$.

\begin{defin}
\label{D:CosimplicialKnots} Let $\Kdot(I,I^n)$ be the collection of spaces $\{C_{\del} \langle p, I^n \rangle \}_{p=0}^{\infty}$
with cofaces given by the doubling maps
$$
d^i\colon C_{\del}\langle p, I^n\rangle  \longrightarrow C_{\del}\langle p+1, I^n\rangle, \ \ 0\leq i\leq p+1,
$$
which are  given by either inserting $y_0$ or $y_1$ (for $d^0$ and $d^{p+1}$) or repeating the $i^{th}$ configuration point (for $d^i$, $i\neq 0,p+1$), and are on $(S^{n-1})^{{p+1 \choose 2}}$ given by repeating all vectors which are the image of $\phi_{ik}$ for any $k$. The vector associated to a repeated point is taken to be $\vec{e}$.   The codegeneracies
$$
s^i\colon C_{\del}\langle p+1, I^n\rangle  \longrightarrow C_{\del}\langle p, I^n\rangle,   \ \ 0\leq i\leq p,
$$
are given by deleting the $i^{th}$ point in the configuration and relabeling appropriately.

\end{defin}

\begin{thm}\label{T:StagesCosimplicialModel} 
 For $n\geq 3$,
\begin{itemize}
\item  $\Kdot(I, I^n)$ is a cosimplicial space \cite[Corollary 4.22]{S:TSK};
\item  $\Tot^{j}\Kdot(I,I^n)\simeq T_{j}\K(I,I^n)$ \cite[Lemma 5.19 and Proposition 6.4]{S:TSK}.
\end{itemize}
\end{thm}

Combining this with the single variable version of \refT{MultiConvergence}, one has

\begin{thm}\label{T:KnotsCosimplicialModel}
For $n>3$, the $\Tot$ tower for $\Kdot(I,I^n)$ converges to $\K(I,I^n)$, i.e.~there is an equivalence $$\Tot\Kdot(I,I^n)\simeq \K(I,I^n).$$
\end{thm}

This cosimplicial model and its Bousfield-Kan spectral sequences have been used extensively in the past several years.  For example, $\Kdot(I,I^n)$ was used for showing the rational collapse of the Vassiliev cohomology spectral sequence for $\K(I,I^n)$, $n>3$, in \cite{LTV:Vass}.  The second author also used $\Kdot(I,I^n)$ in the case $n=3$ to study finite type knot invariants from the point of view of calculus of functors \cite{V:FTK}.  As mentioned in the introduction, the authors plan to use the results of this paper to obtain corresponding results for spaces $\LK_m(\vec{I},I^n)$ and $\HLK_m(\vec{I},I^n)$.

\begin{rem}  All of the results so far have their analog when $I^n$ is replaced by a more general manifold $N$.  One technical complication in this case is that the preferred vector $\vec{e}$ has to be replaced by a choice of a vector in the tangent space at each point of $N$.  If $N$ is parallelizable, then a single vector can be chosen for all points of $N$, but otherwise it cannot.  One way to handle this is to attach tangent spheres to each of the points of $C_{\del}\langle p, N \rangle$ and use this extra tangential data to double the points.  This is explained in detail in \cite[Section 4.5]{S:TSK}. One can then also construct a cosimplicial model for the space of immersions (essentially replace $C_{\del}\langle p, N \rangle$ by $N^p$ in \refD{CosimplicialKnots}); the fiber of the two cosimplicial spaces would then model the space $\hofiber(\Emb_{\del}(I,N)\to \Imm_{\del}(I,N))$.

\end{rem}


\subsection{Cosimplicial model  for spaces of string links}\label{S:LinksModel}


To define a cosimplicial model for $\LK_m(\vec{I},I^n)$, we proceed as in the case of $\K(I,I^n)$ by first noting that, in analogy with \refP{PuncturedKnotsConfs}, we have that spaces of punctured string links that model the stages of the Taylor multitower from Section \ref{S:MultiCalc} are configuration spaces.  More precisely, given a tuple $\vec{\jmath}=(j_1,\ldots, j_m)$ of nonnegative integers, let $A^i_0,\ldots, A^i_{j_i}$ be pairwise disjoint closed subintervals of the $i^{th}$ copy of $I$ in $\vec{I}$. Let $\vec{A}_{[\vec{\jmath}\,]}=(A_{[j_1]},\ldots, A_{[j_m]})$, where $A_{[j_i]}=\cup_{k=0}^{j_i}A^i_k$, and let $\vec{I}-{\vec{A}}=(I-A_{[j_1]},\ldots, I-A_{[j_m]})$. 

\begin{prop}\label{P:PuncturedLinksConfs}
There is an equivalence
\begin{equation}\label{E:PuncturedLinksConfs}
\LK_m(\vec{I}-{\vec{A}_{[\vec{\jmath\,}]}})\simeq C_{\del}\Big(\sum j_i, I^n\Big).
\end{equation}
\end{prop}

The proof of \refP{PuncturedKnotsConfs} applies here word for word; the equivalence is again given by retraction of arcs to their midpoints.

As in Section \ref{S:KnotsModel}, consider now the compatification $C_{\del}\langle \sum p_i, I^n\rangle$ of $C_{\del}(\sum p_i, I^n)$.  It is useful now to write $C_{\del}(\sum p_i, I^n)=C_{\del}(p_1+\cdots +p_m,I^n)$ and $C_{\del}\langle \sum p_i, I^n\rangle=C_{\del}\langle p_1+\cdots +p_m, I^n\rangle$. Also recall Definition \ref{D:m-cosimplicial}.

\begin{defin}\label{D:m-CosimplicialModel} Let $LK_m^{\vec{\bullet}}(\vec{I}, I^n)$ be the $m$-cosimplicial space whose value on $[\vec{p}\,]=([p_1], [p_2], ..., [p_m])\in (\Delta)^m$ is $C\langle p_1+\cdots +p_m, I^n\rangle$ and whose cofaces and codegeneracies in the $k^{th}$ direction,
$$
d^{i}_{k}\colon C_{\del}\langle p_1+\cdots+p_k+\cdots+p_m, I^n\rangle \longrightarrow
C_{\del}\langle p_1+\cdots+(p_k+1)+\cdots+p_m, I^n\rangle
$$
and
$$
s^{i}_{k}\colon C_{\del}\langle p_1+\cdots+(p_k+1)+\cdots+p_m, I^n\rangle
\longrightarrow
C_{\del}\langle p_1+\cdots+p_k+\cdots+p_m, I^n\rangle, 
$$
are given by the doubling and forgetting maps from \refD{CosimplicialKnots}.

\end{defin}

Recall from Section \ref{S:m-Cosimplicial} that one can now consider partial multi-totalizations $\Tot^{\vec{\jmath}}LK^{\vec{\bullet}}_m(\vec{I},I^n)$.  Then, generalizing the second part of \refT{StagesCosimplicialModel}, we have

\begin{prop}\label{P:MultiStagesCosimplicialModel}
There is an equivalence
$$
\Tot^{\vec{\jmath}}LK_m^{\vec{\bullet}}(\vec{I},I^n)\simeq T_{\vec{\jmath}}\LK_m(\vec{I},I^n).
$$
\end{prop}

\begin{proof}
This is a straightforward adaptation of Lemma  5.19 and Proposition 6.4 in \cite{S:TSK}.
\end{proof}

As explained in Section \ref{S:m-Cosimplicial}, we can now consider the diagonal cosimplicial space associated to $LK_m^{\vec{\bullet}}(\vec{I},I^n)$,
\begin{equation}\label{E:LinkDiagonal}
(LK^{\vec{\bullet}}_{m})_{diag}(\vec{I},I^n)=\{ C_{\del}\langle pm, I^n  \rangle \}_{p=0}^{\infty}
\end{equation}

where cofaces and codegeneracies double or forget $m$ points at the same time.  To relate this back to the space of links, one should think of these configuration spaces as obtained by punching holes in the $m$ strands at the same time.  After one punching, the link space becomes $C(0, I^n)$ (we can retract everything back to the fixed strand endpoints on $I^{n-1}\times\{0\}$ and $I^{n-1}\times\{1\}$), after two punches we get  $C(m, I^n)$ (retract to endpoints as well as midpoints of $m$ ``loose" arcs), after three punches we get $2m$ loose arcs which retract to $C(2m, I^n)$, and so on.

\begin{rem}
To extend this construction to an arbitrary manifold $N$, some care with doubling vectors would have to be taken, but this can be taken care of one strand at at time following the case of knots.
\end{rem}


\subsection{Cosimplicial model for spaces of homotopy string links}\label{S:HomotopyLinksModel}


Defining a cosimplicial model for $\HLK_m(\vec{I},I^n)$ goes very much the same way as for $\LK_m(\vec{I},I^n)$.  One modification is that, as each strand of the homotopy string link is allowed to pass through itself, configuration points arising from punching holes are allowed to pass through each other if they come from the same strand. The partial configuration spaces $C_{\partial}(\vec{\jmath}-\vec{1},I^n)$ (see \refS{Conventions} for definition) model this situation. Here $\vec{1}=(1,1,\ldots, 1)$.


We wish to establish an analog of \refP{MultiStagesCosimplicialModel} for link maps (\refP{HMultiStagesCosimplicialModel} below).
As usual, let $\vec{I}=(I,\ldots, I)$. For all $i=1$ to $m$, let $S_i\in\mathcal{P}_0[p_i]$ be a finite set and $\vec{S}=(S_1, \ldots, S_m)$. As in \refP{FiniteTower}, for each $i$ let $A^i_0,\ldots, A^i_{j_i}$ be pairwise disjoint nonempty closed subintervals of the $i^{th}$ copy of $I$, and set $A_{S_i}=\cup_{s\in S_i} A^i_{s}$ and $\vec{A}_{\vec{S}}=(A^1_{S_1},\ldots, A^m_{S_m})$.


Let $\vec{j}=(|S_1|,\ldots, |S_m|)$. Then $C_\del({\vec{\jmath}-\vec{1}},I^n)$ can be thought of as the subspace of $\HLK_m(\vec{I}-\vec{A}_{\vec{S}};I^n)$ consisting of all $\vec{f}=(f_1,\ldots, f_m)$ such that for all $i$, $f_i$ is constant on each component of $I-A^i_{S_i}$.  We then have the following analog of \refP{PuncturedLinksConfs}.

\begin{prop}\label{P:PuncturedHomotopyLinksConfs}  There is an equivalence
\begin{equation}\label{E:PuncturedHomotopyLinksConfs}
\HLK_m(\vec{I}-\vec{A}_{\vec{S}};I^n)\simeq C_\del({\vec{\jmath}-\vec{1}},I^n).
\end{equation}
\end{prop}

\begin{proof}
It is easily seen that the map $\HLK_m(\vec{I}-\vec{A}_{\vec{S}};I^n)\rightarrow C_\del({\vec{\jmath}-\vec{1}},I^n)$ which sends $\vec{f}=(f_1,\ldots, f_m)$ to its restriction to the midpoints of the connected components of $\vec{I}-\vec{A}_{\vec{S}}$ defines a retraction, since the connected components of $\vec{I}-\vec{A}_{\vec{S}}$ are contractible.
\end{proof}

We want to compare the diagram associated to the cosimplicial space for $C_\del({\vec{\jmath}-\vec{1}},I^n)$ with that for $\HLK_m(\vec{I}-\vec{A}_{\vec{S}};I^n)$, and so we need to know that there is a map of such diagrams. This is the content of the following statement, whose proof is straightforward.

\begin{prop}\label{P:PuncturedHomotopyLinksConfsNatural}
Let $\vec{S}=(S_1,\ldots, S_m)$ be as above, and define $\vec{S'}=(S_1',\ldots, S_m')$, where $S_i'=S_i\cup k$, and $S_j'=S_j$ for all $j\neq i$. Let $\vec{\jmath}=(|S_1|,\ldots, |S_m|)$, and $\vec{\jmath}\,'=(|S_1'|,\ldots, |S_m'|)$.

We have a commutative diagram.
$$
\xymatrix{
C_\del({\vec{\jmath}-\vec{1}},I^n)\ar[r]\ar[d] & \HLK_m(\vec{I}-\vec{A}_{\vec{S}};I^n)\ar[d]\\
C_\del({\vec{\jmath}\,'-\vec{1}},I^n)\ar[r] & \HLK_m(\vec{I}-\vec{A}_{\vec{S}'};I^n)
}
$$
where the horizontal maps are inclusions as in \refP{PuncturedHomotopyLinksConfs}, the right vertical map is the restriction, and the left vertical map doubles the $i^{th}$ point.  Here $i$ is the number of elements of $S_i$ less than $k$.
\end{prop}

\begin{rem}
The analog of this diagram is not strictly commutative in the case of the space of string links $\LK_m$. It does commute up to homotopy, however, but this complicates matters. See \cite{S:TSK}.
\end{rem}



In analogy with \refD{m-CosimplicialModel}, we have

\begin{defin}\label{m-CosimplicialHomStrLinksModel}
Let $HLK_m^{\vec{\bullet}}(\vec{I},I^n)$ be the $m$-cosimplicial space whose value on $[\vec{p}\,]=([p_1], [p_2], ..., [p_m])\in (\Delta)^m$ is $C_{\del}(p_1,\ldots,p_m;I^n)$ and whose cofaces and codegeneracies in the $k^{th}$ direction,
$$
d^{i}_{k}\colon C_{\del}(p_1,\ldots,p_k,\ldots,p_m; I^n) \longrightarrow
C_{\del}(p_1,\ldots,p_k+1,\ldots,p_m; I^n)
$$
and
$$
s^{i}_{k}\colon C_{\del}(p_1,\ldots,p_k+1,\ldots,p_m; I^n)
\longrightarrow
C_{\del}(p_1,\ldots,p_k,\ldots,p_m; I^n), 
$$
are given by the doubling and forgetting maps.
\end{defin}

\begin{rem}
Notice that we do not need to compactify these partial configuration spaces.  This is because points coming from single strand are allowed to pass through each other. That is, the image of the doubling map given by repeating a point already lies in the partial configuration space $C_{\del}(p_1,\ldots,p_k+1,\ldots,p_m; I^n)$.   This is not the case with ordinary configuration spaces.
\end{rem}

We now have an analog of Proposition \ref{P:MultiStagesCosimplicialModel} for homotopy string links. 

\begin{prop}\label{P:HMultiStagesCosimplicialModel}
There is an equivalence
$$
\Tot^{\vec{\jmath}}HLK_m^{\vec{\bullet}}(\vec{I},I^n)\simeq T_{\vec{\jmath}}\HLK_m(\vec{I},I^n).
$$
\end{prop}

\begin{proof}
By Proposition \ref{P:iteratedtot}, $$\Tot^{\vec{\jmath}}HLK_m^{\vec{\bullet}}(\vec{I},I^n)\simeq \Tot_1^{j_1}\Tot_2^{j_2}\cdots\Tot_m^{j_m}HLK_m^{\vec\bullet}(\vec{I},I^n).$$

On the other hand, Proposition \ref{P:FiniteTower} says

$$T_{\vec{\jmath}}\HLK_m(I_1,\ldots, I_m;I^n)\simeq \underset{(S_1,\ldots, S_m)\in\prod_i\mathcal{P}_0(j_i+1)}{\holim}\, \HLK_m(I-A^1_{S_1},\ldots, I-A^m_{S_m};I^n),$$

and the right hand side can be rewritten as an iterated homotopy limit by equation \eqref{E:HolimProduct}; i.e.~as
$$
\underset{S_1\in\mathcal{P}_0(j_1+1)}{\holim}\,\underset{S_2\in\mathcal{P}_0(j_2+1)}{\holim}\cdots\underset{S_m\in\mathcal{P}_0(j_m+1)}{\holim}\,\HLK_m(I-A^1_{S_1},\ldots, I-A^m_{S_m};I^n).
$$
\refP{PuncturedHomotopyLinksConfs} gives an equivalence
$$
\HLK_m(I-A^1_{S_1},\ldots, I-A^m_{S_m};I^n)\simeq C_{\del}(\abs{S_1}-1,\abs{S_2}-1,\ldots,\abs{S_m}-1;I^n),
$$
and \refP{PuncturedHomotopyLinksConfsNatural} says the equivalence is functorial; that is, it gives a map of diagrams. Finally, repeated application of Lemma \ref{L:totandholim} identifies 
$$
\underset{S_1\in\mathcal{P}_0(j_1+1)}{\holim}\,\underset{S_2\in\mathcal{P}_0(j_2+1)}{\holim}\cdots\underset{S_m\in\mathcal{P}_0(j_m+1)}{\holim}\,C_{\del}(\abs{S_1}-1,\abs{S_2}-1,\ldots,\abs{S_m}-1;I^n)
$$
with
$$
\Tot^{j_1}\Tot^{j_2}\cdots\Tot^{j_m}\,HLK_m^{\vec{\bullet}}(\vec{I},I^n).$$
\end{proof}

The diagonal cosimplicial space associated to $HLK_m^{\vec{\bullet}}(\vec{I},I^n)$ is
\begin{equation}\label{E:HomotopyLinkDiagonal}
(HLK_m^{\vec{\bullet}})_{diag}(\vec{I},I^n)=\{ C_{\del}(p, p, ..., p; I^n) \}_{p=0}^{\infty},
\end{equation}
This is completely analogous to the diagonal cosimplicial space for space of string links.

\begin{rem}
Generalization of this construction to an arbitrary $N$ is easier than in the case of links since one does not need to keep track of tangent vectors in this case.
\end{rem}
%




\subsection{Bousfield-Kan spectral sequences for spaces of links and their convergence properties}\label{S:LinksSS's}


We now analyze Bousfield-Kan spectral sequences associated to $(LK_m^{\vec{\bullet}})_{diag}(\vec{I},I^n)$ and $(HLK_m^{\vec{\bullet}})_{diag}(\vec{I},I^n)$ from equations \eqref{E:LinkDiagonal} and \eqref{E:HomotopyLinkDiagonal}.  Here it is important that the ambient manifold is $I^n$ since we only have a good handle on the homotopy and cohomology of configuration spaces in Euclidean space and not in an arbitrary manifold.  However, it is still true that the spectral sequences from Theorems \ref{T:MainThmLinks} and \ref{T:MainThmHomotopyLinks} exist for any manifold $N$ although we do not know anything about their convergence properties.

\begin{thm}\label{T:MainThmLinks}  For $n> 3$, there exist homotopy and, for any field of coefficients, cohomology spectral sequences whose first pages are given by 
$$
E^{1}_{-p,q}= \pi_q(C_{\del}\langle pm, I^{n}\rangle)\bigcap\Big(\bigcap_{i=0}^{p}ker(s^{i}_{*})\Big)$$
and
$$   E_{1}^{-p,q}=H^q(C_{\del}\langle pm, I^{n}\rangle/\Big(\sum_{i=0}^{p-1} im((s^i)^*)\Big)
$$
and which converge to $\pi_*(\Tot (LK_m^{\vec{\bullet}})_{diag}(\vec{I},I^n))$ and $H^*(\Tot (LK_m^{\vec{\bullet}})_{diag}(\vec{I},I^n))$, respectively.
\end{thm}

From \refP{MultiStagesCosimplicialModel} and \refT{MultiConvergence}, we immediately deduce 
\begin{cor}\label{C:ConvergenceToLinks}  
The above spectral sequences converge to the homotopy and cohomology of $\LK_m(\vec{I},I^n)$ for $n>3$.
\end{cor}


The proof of \refT{MainThmLinks} will follow from Propositions \ref{P:HomotopyConvergence} and \ref{P:CohomologyConvergence} below. 

\begin{prop}\label{P:HomotopyConvergence}
The connectivity of the homotopy fibers
$$
L_j (LK_m^{\vec{\bullet}})_{diag}(\vec{I},I^n)=\hofiber\big(\Tot^j (LK_m^{\vec{\bullet}})_{diag}(\vec{I},I^n)\longrightarrow \Tot^{j-1} (LK_m^{\vec{\bullet}})_{diag}(\vec{I},I^n)\big)
$$
 of the $\Tot$ tower for $(LK_m^{\vec{\bullet}})_{diag}(\vec{I},I^n)$ increases with $j$ for $n>3$.  More precisely, $L_j (LK_m^{\vec{\bullet}})_{diag}(\vec{I},I^n)$ is $(j-1)(n-3)$-connected.
\end{prop}

\begin{rem}
This is true with $N$ in place of $I^n$. As for the choice of basepoint, the equivalence $\Tot^{\vec{\jmath}}LK_m^{\vec{\bullet}}(\vec{I},I^n)\simeq T_{\vec{\jmath}}\LK_m(\vec{I},I^n)$ from \refP{MultiStagesCosimplicialModel} gives a basepoint to $\Tot^{j} (LK_m^{\vec{\bullet}})_{diag}(\vec{I},I^n)$ for all $j$ since $\LK_m(\vec{I},I^n)$ is based by the unlink.
\end{rem}

\begin{proof}
This follows immediately from Propositions \ref{P:multiembcartesian} and \ref{P:CubeConvergence} after two things are noted. The first is that the application of $\Omega^j$ to a space decreases its connectivity by $j$. The second is that the connectivity estimates we are using involve a cubical diagram of ordinary configurations, whereas our cosimplicial models involve compactified configuration spaces. However, the inclusion of the ordinary configuration space into the compactified one is an equivalence \cite[Theorem 5.10]{S:TSK}, and this equivalence is functorial so that it gives a map of cosimplicial diagrams, and hence the connectivity estimates are the same for both diagrams.
\end{proof}

\begin{prop}\label{P:CohomologyConvergence}
The $E_{1}^{-p,q}$ term of the cohomology spectral sequence from \refT{MainThmLinks} vanishes for $q<p(n-1)/2$ and for $q>(pm-1)(n-1)$.
\end{prop}

\begin{proof}

First recall from \cite{Cohen:Homology} that the 
cohomology of $C_{\del}\langle pm, I^{n}\rangle\simeq C(pm, \R^{n})$, $n\geq 3$, is  generated by classes $\alpha_{ij}$, $1\leq i< j\leq pm$, in dimension $n-1$.  (The relations can be found, for example, in \cite[Section 6.4]{S:TSK}; we will not need them here.)

%

A convenient description of the cohomology classes  is via chord diagrams on $pm$ vertices, where an edge between vertices $i$ and $j$ represents $\alpha_{ij}$.\footnote{This point of view has been used in \cite{LT:Graphs, S:Pairing} and provides the bridge to the theory of finite type invariants (see \cite{V:FTK}).}  More precisely, a cohomology class of degree $k(n-1)$ in $C(pm, \R^{n})$ can  be thought of as a  diagram of  $m$ horizontal line segments (representing $m$ link strands), each with $p$ points labeled $1, 2, ..., p$ on it, with  $k$ chords distributed in some fashion among the points.  Each chord can connect points on same or different line segments.

With this point of view, cofaces and codegeneracies in $(LK_m^{\vec{\bullet}})_{diag}(\vec{I},I^n)$ either add or delete a column of $p$ points labeled the same way on the $m$ strands.  In more detail, the 
$i^{th}$ codegeneracy can be thought of as a map
\begin{align*}
s^i\colon C_{\del}\langle pm, I^{n}\rangle &  \longrightarrow C_{\del}\langle p(m-1), I^{n}\rangle,\ \ \ \ 0\leq i\leq m-1 \\
(x_1, x_2, ..., x_m) & \longmapsto (y_1, y_2, ..., y_{m-1})
\end{align*}
where each $x_j$ (resp.~$y_j$) is a configuration of $p$ (resp.~$p-1$) points representing the $j$th column in the chord diagram, and where $y_j=x_j$ for $j\leq i$ and $y_j=x_{j+1}$ for $j>i+1$.  Since configuration $x_{i+1}$ does not appear in the image of $s^i$, it follows that a chord diagram representing an element of  $H^*(C_{\del}\langle pm, I^{n}\rangle)$ can only be in the image of 
$$(s^i)^*\colon H^*(C_{\del}\langle p(m-1), I^{n}\rangle)\longrightarrow  H^*(C_{\del}\langle pm, I^{n}\rangle)
$$
 if it had no chords in the $(i+1)$st column.  The normalization module $\sum_{i=0}^{p-1} im((s^i)^*)$ is thus generated by all chord diagrams which have no chords in at least one column.  It then follows that $E_{1}^{-p,q}$ is generated by all chord diagrams which have at least one chord in every column.  This, however, is only possible if $q\geq p(1-n)/2$ (keep in mind that $p$ is positive).  In other words, the spectral sequence has a lower vanishing line of slope $(1-n)/2$.
 
For the upper vanishing line, recall that the Poincar\'{e} polynomial for $C(pm, \R^n)$ is
\begin{equation}\label{E:PoincarePoly}
P_{p}(t)=(1+t^{n-1})(1+2t^{n-1})\cdots(1+(pm-1)t^{n-1}).
\end{equation}
The top cohomology group of $C(pm,\R^n)$ thus occurs in degree $(pm-1)(n-1)$, and so the $E_1$ term vanishes for $q>(pm-1)(n-1)$.
\end{proof}

\begin{rem}
\refP{CohomologyConvergence} is analogous to Corollary 7.7 in \cite{S:TSK} which gives the vanishing lines for the cohomology spectral sequence for $\K(I,I^n)$, $n>3$.
\end{rem}

\begin{proof}[Proof of Theorem \ref{T:MainThmLinks}]
The fact that the spectral sequences for $(LK_m^{\vec{\bullet}})_{diag}(\vec{I},I^n)$ start with the desired first page is a restatement of equations \eqref{E:FirstPageHomotopy} and \eqref{E:FirstPageCohomology}.

By \refP{HomotopyConvergenceCondition}, the homotopy spectral sequence converges to $\pi_*(\Tot (LK_m^{\vec{\bullet}})_{diag}(\vec{I},I^n))$ after combining \refP{HomotopyConvergence} with \refP{CubeConvergence}.  
Similarly, from \refT{CohomologyConvergenceCondition} and Proposition \ref{P:CohomologyConvergence} we deduce that the cohomology spectral sequence converges to the cohomology of $\Tot (LK_m^{\vec{\bullet}})_{diag}(\vec{I},I^n)$.  This is true since the spaces $C_{\del}\langle pm, I^{n}\rangle$ are connected, of finite type, and since there exists a steep vanishing line in the spectral sequence for $n>3$.
\end{proof}

A statement analogous to (but weaker than) \refT{MainThmLinks} holds for $HLK_m^{\vec{\bullet}}(\vec{I},I^n)$.

\begin{thm}\label{T:MainThmHomotopyLinks}  For $n\geq 3$, there exist homotopy and, for any field of coefficients, cohomology spectral sequences whose first pages are given by 
$$
E^{1}_{-p,q}= \pi_q(C_{\del}(p,p, ..., p; I^{n}))\bigcap\Big(\bigcap_{i=0}^{p}ker(s^{i}_{*})\Big)$$
and
$$   E_{1}^{-p,q}=H^q(C_{\del}(p,p, ..., p; I^{n})/\Big(\sum_{i=0}^{p-1} im((s^i)^*)\Big)
$$
respectively (as usual, here $p$ appears $m$ times).  The cohomology spectral sequence converges to the cohomology of $\Tot HLK_m^{\vec{\bullet}}(\vec{I},I^n)$ for $n>3$.

\end{thm}

\begin{rem}\label{R:RemHLinksConvergence}
The reason we do not have convergence for the homotopy spectral sequence in this case is that we now do not have an analog of \refP{HomotopyConvergence} for homotopy string links. Proposition \ref{P:multiembcartesian} is what gives us the convergence of the homotopy spectral sequence, which itself is a simple consequence of  \refP{embcartesian}. The latter uses the Blakers-Massey Theorem, and the isotopy extension theorem, which says that the restriction map $\Emb(K_2,N)\rightarrow\Emb(K_1,N)$ induced by an inclusion of compact sets $K_1\subset K_2$ is a fibration. This makes the identification of the fibers of the maps in Proposition \ref{P:embcartesian}, and the application of the Blakers-Massey Theorem an easy task. For spaces of link maps, it is not true that the restriction map is a fibration, and so any attempt at computing connectivity estimates for spaces of link maps will require a finer analysis, and it is an open question what these estimates might be.  For the same reason, we do not have an analog of \refC{ConvergenceToLinks} since we do not have an analog of \refT{MultiConvergence} for the space of homotopy links.
\end{rem}

As in the case of the cohomology spectral sequence for links, \refT{MainThmHomotopyLinks} follows immediately from the following

\begin{prop}  The vanishing conditions from 
Proposition \ref{P:CohomologyConvergence} also hold for the cohomology spectral sequence from \refT{MainThmHomotopyLinks}.
\end{prop}

\begin{proof}  From \cite[Corollary 5.6]{LS:CohArrang}, we have that $H^*(C_{\del}(p,p, ..., p; I^{n}))\cong H^*(C(p,p, ..., p; \R^{n}))$ is generated by classes $\alpha_{x^a_lx^b_k}$ in dimension $n-1$, where $1\leq l,k\leq p$, $1\leq a,b\leq m$, $a\neq b$.  The set of generators is thus a subset of the generators of $H^*(C(pm, \R^{n}))$ given by diagrams with no chords whose both endpoints lie on the same horizontal line segment. The codegeneracies still impose the same condition on the normalized $E_1$ page of the spectral sequence so that the lower vanishing line has the same slope as in the case of links.

For the upper vanishing line, notice that the relations given in \cite[Corollary 5.6]{LS:CohArrang}  make it possible to rewrite an element of the top class in $H^*(C(p,p, ..., p; \R^{n}))$ as a sum of trees on $pm$ vertices.  Since such trees have $pm-1$ edges, each representing a generator in dimension $n-1$, the top cohomology dimension is $(pm-1)(n-1)$, just as in the case of the configuration space $H^*(C(pm, \R^{n}))$.  Thus the upper vanishing line is the same as in Proposition \ref{P:CohomologyConvergence}.
\end{proof}


\section{The case of braids}\label{S:Braids}


Here we briefly note that most constructions from this paper could have been done for braids in addition to string links and homotopy string links. First recall that a \emph{(pure) braid on $m$ strands in $I^n$} is an embedding of $\coprod_{m}I$ in $I^n$ with boundary conditions as in \refD{linkspaces} with the additional conditions that the  last component of the derivative along each of the strands is constant.  We will denote by $\BR_m(\vec{I}, I^n)$ the space of braids.
Multivariable calculus applies to this functor as well, so that we could construct multi-stages $T_{\vec{\jmath}}\BR_m(\vec{I},I^n)$ using diagrams of punctured braids, deduce from this a multi-cosimplicial model, and obtain the associated diagonal cosimplicial space $(BR_m^{\vec{\bullet}})_{diag}(\vec{I},I^n)$.

However, a quicker way to get at this cosimplicial space is as follows:  First observe that a braid in $\BR_m(\vec{I}, I^n)$ can be thought of as an embedding with same boundary conditions but which is also \emph{level-preserving}, i.e.~a point $t\in I$ is mapped to the same slice of $I^n$ under each of the component embeddings.  In other words, a braid is a path of $m$ configuration points in $I^{n-1}$  (in fact, this is often taken as the definition of a braid).  However, because the parametrization is now uniform for all strands, this is a functor of a single variable, namely of open subsets of $I$ and we can write
$$
\BR_m(I,I^n)\simeq \Omega C(m, I^{n-1}).
$$
A little care about basepoints should be taken here.  Namely, we have here implicitly identified $\BR_m(I,I^n)$ with the homotopy equivalent space of ``closed'' braids obtained by identifying $I^{n-1}\times \{0\}$ with $I^{n-1}\times \{1\}$.  Since the embeddings were fixed at either end of $I^n$, this identification produces a natural basepoint in $C(m, I^{n-1})$.

A cosimplicial model $BR^{\bullet}_m(I, I^n)$ can be constructed for this functor as in \refS{KnotsModel} as follows.  To first construct the stages of the Taylor tower, holes are punched in $I$, but in the image of the braid, this means that the holes are always punched in the same place in all the braid strands.  If the resulting arcs are then always retracted to the midpoints it follows that
$$
\BR_m\left(I-A_{[j]}, I^n\right)\simeq 
(C(m, I^{n-1}))^{j}
$$ 
where as before $A_{[j]}=A_0\cup\cdots\cup A_j$.  This is because points on different strands at each level lie in the same plane $\{t\}\times I^{n-1}$ and there is no interaction between the planes for different values of $t$.
Therefore the cosimplicial space consists of copies of configurations in $I^{n-1}$.  More precisely,
\begin{equation}\label{E:BraidsCosimplicial}
BR^{\bullet}_m(I,I^n)=\{(C( m,I^{n-1}))^p \}_{p=0}^{\infty}.
\end{equation}
Of course, one could arrive at this cosimplicial space by continuing to regard braids as a multivariable functor and observing that the positive derivative condition implies that there is a homotopy equivalence
\begin{equation}\label{E:BraidsDiagonal}
\left((BR_m^{\vec{\bullet}})_{diag}(\vec{I},I^n)\right)^{[p]}\simeq 
(C(m, I^{n-1}))^{p}.
\end{equation}
It is thus  easily seen that
$$
\Tot BR^{\bullet}_m(I,I^n)  \simeq  \Tot (BR_m^{\vec{\bullet}})_{diag}(\vec{I},I^n).
$$ 
For the second part of the picture, recall that a standard example of a cosimplicial space is the cosimplicial model for $\Omega X$, the loop space of a space $X$ (also known as the \emph{geometric cobar construction on $X$}) defined as follows. 

Consider the cosimplicial space $X^{\bullet}$ whose $(p+1)$st space is $X^{p}$ ($X^{0}=*$ is the first space).  Now define the cofaces by
\begin{align}
d^i\colon X^{p} & \longrightarrow X^{p+1} \notag \\
(x_1, x_2, ..., x_p) & \longmapsto
\begin{cases}\label{E:LoopCofaces}
(*, x_1, x_2, ..., x_p),            & i=0;  \\
(x_1, x_2, ...,x_i, x_i, ..., x_p), & 1\leq i\leq p;  \\
(x_1, x_2, ..., x_p, *),            & i=p+1;  \\
\end{cases}
\end{align}
and codegeneracies by
\begin{align*}
s^i\colon X^{p} & \longrightarrow X^{p-1} \\
(x_1, x_2, ..., x_p) &  \longmapsto
(x_1, x_2, ...,x_{i}, x_{i+2}, ... x_p).
\end{align*}
It is then not hard to show 
 that there is a homeomorphism
\begin{equation}\label{E:LoopModel}
\Omega X\cong \Tot X^{\bullet}.
\end{equation}

We then have
\begin{prop}\label{P:TotForBraids}  For $n\geq 2$, 
the cosimplicial spaces from equations \eqref{E:BraidsCosimplicial} and \eqref{E:BraidsDiagonal} are the cobar construction for $\Omega C(m, I^{n-1})$.  In other words,
$$
\Tot BR^{\bullet}_m(I,I^n)  \simeq  \Tot (BR_m^{\vec{\bullet}})_{diag}(\vec{I},I^n)\simeq \Omega C(m, I^{n-1}).
$$
\end{prop}
\begin{proof}
Equation \eqref{E:BraidsCosimplicial} already shows that the spaces in $BR^{\bullet}_m(I,I^n)$ (or $(BR_m^{\vec{\bullet}})_{diag}(\vec{I},I^n)$) are what they should be, namely powers of $C(m,I^{n-1})$.  It is also clear that our doubling and forgetting maps are precisely those from the cobar construction.  Note that, because we have made the identification $I^{n-1}\times\{0\}\sim I^{n-1}\times\{1\}$, the first and last doubling map in $BR^{\bullet}_m(I,I^n)$ correspond to inserting the basepoint at the beginning or the end of the tuple of points, as in the cobar construction.
\end{proof}

\bibliographystyle{amsplain}


\bibliography{/Users/ivolic/Desktop/Papers/Bibliography}

\end{document}